\documentclass[a4paper]{amsart}
\usepackage[latin1]{inputenc}
\usepackage[T1]{fontenc}
\usepackage[english]{babel}
\usepackage{amsmath}
\usepackage{amscd}
\usepackage{graphicx}
\usepackage[all]{xy}
\usepackage{amsthm}
\usepackage{amssymb,url}
\usepackage[charter]{mathdesign}
\usepackage{bbm}
\usepackage{ae}
\usepackage[usenames,dvipsnames]{color}

\newtheorem{thm}{Theorem}[section]
\newtheorem{thm*}{Theorem}[]

\newtheorem{cor}[thm]{Corollary}
\newtheorem{prop}[thm]{Proposition}

\newtheorem{prop*}[thm*]{Proposition}
\newtheorem{lem}[thm]{Lemma}

\newtheorem{lem*}[thm*]{Lemma}

\theoremstyle{remark}
\newtheorem{remark}[thm]{Remark}
\newtheorem{remark*}[thm*]{Remark}

\theoremstyle{definition}
\newtheorem{deef}[thm]{Definition}
\newtheorem{deef*}[thm*]{Definition}

\newtheorem*{convention}{Convention}

\newcommand{\R}{\mathbbm{R}}
\newcommand{\C}{\mathbbm{C}}

\newcommand{\Z}{\mathbbm{Z}}
\newcommand{\Q}{\mathbbm{Q}}
\newcommand{\N}{\mathbbm{N}}

\newcommand{\id}{\, \mathrm{id}}
\newcommand{\rd}{\mathrm{d}}

\newcommand{\Ko}{\mathcal{K}}

\newcommand{\He}{\mathcal{H}}

\newcommand{\ind}{\mathrm{i} \mathrm{n} \mathrm{d} \,}

\newcommand{\Hom}{\mathrm{Hom}}

\renewcommand{\epsilon}{\varepsilon}
\renewcommand{\phi}{\varphi}
\renewcommand{\mod}{\;\mathrm{mod}\,}

\newcommand{\ch}{\mathrm{ch} \,}

\newcommand{\e}{\mathrm{e}}
\newcommand{\tra}{\mathrm{t}\mathrm{r}}

\title{Applying geometric $K$-cycles to fractional indices}
\author{Robin J. Deeley, Magnus Goffeng}

\address{Robin J. Deeley,\newline
\indent Mathematisches Institut, Georg-August Universit\"at G\"ottingen\newline
\indent Bunsenstrasse 3-5\newline
\indent D-37073 G\"ottingen\newline
\indent Germany\newline
\newline 
\indent Magnus Goffeng,\newline
\indent Institut f\"ur Analysis, Leibniz Universit\"at Hannover\newline
\indent 301 67 Hannover\newline
\indent Germany\newline}

\begin{document}

\begin{abstract}
A geometric model for twisted $K$-homology is introduced. It is modeled after the Mathai-Melrose-Singer fractional analytic index theorem in the same way as the Baum-Douglas model of $K$-homology was modeled after the Atiyah-Singer index theorem. A natural transformation from twisted geometric $K$-homology defined by Wang to the new geometric model is constructed. The analytic assembly mapping to analytic twisted $K$-homology in this model is an isomorphism for torsion twists on a finite $CW$-complex. For a general twist on a smooth manifold the analytic assembly mapping is a surjection. Beyond the aforementioned fractional invariants, we study $T$-duality for geometric cycles.
\end{abstract}
\maketitle

\vspace{-1mm}
\Large
\section*{Introduction}
\normalsize

One of the motivating problems for the Atiyah-Singer index theorem was the search for an explanation for the integrality of the $\hat{A}$-genus of a spin manifold. The answer is that the $\hat{A}$-genus is the Fredholm index of the Dirac operator associated with the spin structure. Mathai-Melrose-Singer asked a similar question in their paper \cite{mms}.  Namely, is there an index type interpretation of the $\hat{A}$-genus when there is no spin structure? The answer to this question required projective pseudo-differential operators and a pseudo-differential calculus with symbols given by sections of an Azumaya bundle.  These operators have similar micro local structure as usual pseudo-differential operators, but the space of smoothing operators in this calculus does not form an algebra. However, they admit a tracial functional. The trace defines an invariant of elliptic elements in the calculus known as the fractional analytic index. The $\hat{A}$-genus of an oriented manifold was proven to equal the fractional analytic index of a projective Dirac operator in \cite{mms}. 

The proof of the Atiyah-Singer index theorem from \cite{atiyahsingerone} uses $K$-theory as a container for symbols of elliptic operators and the index theorem can be stated as the equality of two different types of indices, one topological and the other analytic. The topological index is a push-forward to a point, so even though $K$-theory is a generalized cohomology theory, it is used as a homology theory. To remedy this limitation, Atiyah constructed the analytic $K$-homology of a topological space in \cite{atiyahglobal}.  This construction was later generalized much further by Kasparov \cite{kasparov} to give $KK$-theory. Even though an elliptic operator easily gives a class in analytic $K$-homology, it is often difficult to extract explicit index theoretic information from such data. 

In the same spirit that integer homology is constructed from simple combinatorial data such as simplicial complexes, Baum-Douglas \cite{baumdtva} introduced geometric $K$-homology which is built on data that have well understood index theory, spin$^c$-manifolds. From a geometric $K$-cycle one can assemble an abstract elliptic operator in analytic $K$-homology. This completed the picture as the Atiyah-Singer index theorem can be derived from a certain commutative diagram involving the assembly mapping; for more details, see \cite{baumdtva}. Geometric $K$-homology with variations has proved to be useful in various other index problems: the index problem for the Heisenberg calculus \cite{baumvanerp}, $\Z/k\Z$-manifolds \cite{deeleyzkone, deeleyzktwo} and $\R/\Z$-index theory \cite{deeleyrmodz, higsonroeEtaAssemRig, lott}.

The aim of this paper is to develop a construction similar to Baum-Douglas geometric $K$-homology for the fractional analytic index theorem of Mathai-Melrose-Singer.  The resulting theory is refered to as projective $K$-homology. A projective cycle $(M,E,\phi)$ over a space $X$ with a twist, i.e. a principle $PU(\He)$-bundle $P\to X$ for a separable Hilbert space $\He$, consists of three parts: a torsion twist $M$ on a spin$^c$-manifold, a twisted $K$-theory class $E$ and a mapping $\phi$. The spin$^c$-manifold with a twist $M$ is a principal $PU(k)$-bundle for some $k$ over the spin$^c$-manifold $PU(k)\backslash M$. The element $E$ is an element of $K$-theory with the twist associated with $M$ over $PU(k)\backslash M$. Finally, the mapping $\phi$ respects the twist of $X$ and $M$ in the sense that it (after possible Morita equivalences) is an equivariant mapping $M\to P$; the precise definition is given in Definition \ref{projectivegeometriccycles}. After defining a suitable equivalence relation we obtain a group $K_*^{proj}(P)$, see Definition \ref{theprojectivekhomology}. The assembly mapping from projective $K$-homology to twisted analytic $K$-homology can be constructed more or less in the usual way.  The manifold $PU(k)\backslash M$ carries a spin$^c$-Dirac operator that one twists by $E$. The resulting class is then pushed forward using $\phi$. When $X$ is a finite $CW$-complex and the twist is torsion, the assembly mapping is an isomorphism (see Theorem \ref{analyticassembly}). 

Examples of projective cycles naturally occur in various areas. When $T$-dualizing a circle bundle, one can construct projective cycles on the $T$-dual twisted space from geometric cycles on the total space of a circle bundle. We prove in Proposition \ref{tdualprop} that $T$-dualizing geometric cycles into projective cycles produces a monomorphism on classes in general and an isomorphism if the circle bundle is flat. The construction implies that the assembly mapping is a surjection on the $T$-dual of a circle bundle. Our main example is that of the clutching construction for elliptic projective pseudo-differential operators. The clutching construction can be performed for torsion twists and maps the twisted $K$-theory of the cotangent bundle, the container for $K$-theory data of symbols of elliptic projective pseudo-differential operators, to the projective $K$-homology.  

Our intention with this geometric model is two-fold.  Not only is it isomorphic to twisted analytic $K$-homology (if the twist is torsion), but the cycles used to define it are of a rather general form.  As such, these cycles appear naturally in a number of applications.  In \cite{wangcycles} (also see \cite{thomcareywang}), a geometric model for twisted $K$-homology is developed using cycles based on considerations from mathematical physics (in particular, D-branes).  As such, the cycles are certainly naturally defined and have also found a number of applications (e.g., the foliation index theorem in \cite{wangcycles}).  On the other hand, a number of constructions considered here seem to require the flexibility offered by projective cycles; a prototypical example is the clutching construction discussed in the previous paragraph.

There is a natural way to obtain a projective cycle from a twisted geometric cycle. This association produces a natural transformation from twisted geometric $K$-homology to projective $K$-homology that we call the geometric modification mapping (see Theorem \ref{geometricmodification}). We relate the geometric modification mapping to the assembly mapping in twisted geometric $K$-homology (see \cite{wangcycles}). In particular, for smooth manifolds, the analytic assembly mapping on projective $K$-homology is surjective and the geometric modification mapping is injective (see Corollary \ref{assemblycorollary}).

Projective $K$-homology for torsion twists fits very well with fractional invariants. Every projective cycle for a torsion twist has a fractional projective index which only depends on its class in projective $K$-homology. Similarly, a twisted geometric cycle has a fractional geometric index only depending on the class in twisted geometric $K$-homology; in the picture of a $D$-brane, this can be thought of as a total charge. We prove in Theorem \ref{geoprojfrac} that the fractional projective index, the fractional geometric index and the fractional analytic index are all equal. The construction of fractional indices from geometric $K$-cycles, both projective and twisted, unfortunately only makes sense for torsion twists. Precisely, in Subsection \ref{nontorsionproblems}, we show that the de Rham class of the twist's Dixmier-Douady class gives an obstruction to the bordism invariance of the fractional index.

Beyond the above mentioned assemble of twisted analytic cycles from projective cycles, there are other analytic realizations of our cycles in the case of a torsion twist. To give a $K$-homological description of the fractional analytic index from a more analytic picture, we use rational geometric $K$-homology as a stepping stone.  As the reader may recall (see \cite{blackadar}), the rational $K$-homology of $X$ can be realized via $KK(C(X),D)$ where $D$ is a UHF-algebra with $K_0$-group given by the rationals.  Moreover, this $KK$-theory group can be realized using Baum-Douglas type cycles (see for example \cite{wal}); we refer to such a cycle as a rational geometric $K$-cycle.  Since the twisted Chern character is a rational isomorphism for a torsion twist, we can in a direct fashion construct a rational geometric $K$-cycle from a projective cycle (see Definition \ref{geofromproj}). The fractional index of a projective cycle coincides with the index constructed from its rationalization as a rational $K$-homology class, i.e., as an index with values in the $K$-theory of the aforementioned UHF-algebra. 

Let us make a few remarks about torsion assumptions on the twist. The definition of projective $K$-homology and constructions of examples work under no torsion assumption on the twist. However, our main results concerning the assembly mapping and fractional invariants are proved only for torsion twists. It is at this point not clear if projective $K$-homology is the correct model for non-torsion twists, but we include its definition due to the interesting examples. Another geometric model, one that fits well also with non-torsion twists, can be found in \cite{BCW}. A number of constructions related to those considered in this paper are discussed in \cite{BCW}. While there is some overlap between our results and those found in \cite{BCW}, we have concentrated more on specific applications while they concentrate more on the general theory. As such, these papers are, in fact, rather complementary.

The paper is organized as follows; in the first section, we recall some known concepts of projective vector bundles and its relations to twisted $K$-theory. Some results in this section are formulated in a slightly new way to fit better with the rest of this paper. In Section \ref{sectionprojectivekhomology}, we introduce the projective cycles and projective $K$-homology and prove that the assembly mapping is an isomorphism for torsion twists. We consider examples coming from $T$-duality and the clutching construction in Section \ref{sectionexamples}. The projective $K$-homology is compared to twisted geometric $K$-homology in Section \ref{sectiongeometricmodification}. In the fifth and final section of the paper, fractional invariants are considered. 

The reader familiar with twisted $K$-theory may skip the first section, only to return to it for notation. Section \ref{sectionprojectivekhomology} contains the main results of the paper. Section \ref{sectionexamples} and \ref{sectiongeometricmodification} can be read independently of each other. The material in Section \ref{fractionalsection} is motivated by the clutching example from Section \ref{sectionexamples}, but is independent of the other parts of  Section \ref{sectionexamples} and \ref{sectiongeometricmodification}.

\Large
\section{Principal $PU(n)$-bundles and projective bundles}
\normalsize

\subsection{The group $PU(\infty)$}

The group of projective unitaries in $n$ dimensions is defined as the group $PU(n):=U(n)/U(1)$, it forms a simple Lie group. The group $PU(n)$ can also be realized as $SU(n)/Z(SU(n))$, the center $Z(SU(n))$ of $SU(n)$ is the finite subgroup consisting of multiples of the identity matrix with $n$th roots of unity. If $n,n'\in \N$ there is a Lie group homomorphism
 \[i_{n,n'}:U(n)\times U(n')\to U(nn'),\;(x,y)\mapsto x\otimes y.\]
Here we identify $U(nn')=U(\C^n\otimes \C^{n'})$. This homomorphism acts as multiplication on the centers. In particular, it induces a homomorphism $PU(n)\times PU(n')\to PU(nn')$. Another consequence is that if $n|n'$ we obtain a canonical mapping $PU(n)\to PU(n')$ in this way that factors over the mapping $PU(n)\to PU(n)\times PU(n'/n)$, $x\mapsto (x,1)$. We will denote $PU(\infty):=\varinjlim PU(n)$.

If $P$ is a locally compact Hausdorff space with a $PU(n)$-action we say that $P$ is a $PU(n)$-space. All actions are tacitly assumed to be left actions. Let us remark that since $PU(n)$ is a compact Lie group, the Slice Theorem implies that if $P$ is a manifold with a free smooth $PU(n)$-action then the quotient $PU(n)\backslash P$ is smooth and the mapping $P\to PU(n)\backslash P$ is a principal $PU(n)$-bundle. We say that a $U(l)$-principal bundle $P^U$ over a $PU(n)$-space $P$ is $PU(n)$-invariant if the lift of the $PU(n)$-action to $P^U$ commutes with the $U(l)$-action on $P^U$. If for instance $P\to PU(n)\backslash P$ is a principal bundle, there is a one to one correspondence between $PU(n)$-invariant $U(l)$-principal bundles on $P$ and $U(l)$-principal bundles on $PU(n)\backslash P$.

\begin{deef}
\label{puinftymorphism}
A $PU(\infty)$-morphism from a $PU(n)$-space $P_1$ to a $PU(n')$-space $P_2$ is a triple $(f,P_1^U,P_2^U)$ where the last two components consist of a $PU(n)$-invariant $U(l)$-bundle $P_1^U\to P_1$ and a $PU(n')$-invariant $U(l')$-bundle $P_2^U\to P_2$ such that $nl=n'l'$ and $f$ is a $PU(nl)$-equivariant continuous mapping 
\[f:PU(nl)\times_{PU(n)\times U(l)}P_1^U\to PU(nl)\times_{PU(n')\times U(l')}P_2^U.\]
If $f$ is a homeomorphism, we say that $f$ is a $PU(\infty)$-isomorphism. We say that $(f,P_1^U,P_2^U)$ is a $PU(\infty)$-diffeomorphism if all involved structures are manifolds and $f$ is a diffemorphism. \end{deef}
\label{projectivemarker}

The composition of two $PU(\infty)$-morphisms $(f,P_1^U,P_2^U)$ and $(g,P_2^U,P_3^U)$, which are $PU(\infty)$-morphisms $P_1\to P_2$ respectively $P_2\to P_3$, is defined as $(g\circ f, P_1^{U}, P_3^U)$. Observe that the composition of two arbitrary $PU(\infty)$-morphisms is not well defined in general. The motivation for introducing this notion of a morphism is its similarity with Morita morphisms. We will return to this later in the context of Azumaya bundles.

We will use the notation $P^{fr}\subseteq P$ for the maximal open subset on which the $PU(n)$-action is free.

\begin{prop}
\label{freecompactneighborhood}
If $P$ is a metric $PU(n)$-space and $K\subseteq P^{fr}$ is compact, $K$ has a $PU(n)$-invariant neighborhood $U$ on which $PU(n)$ acts freely.
\end{prop}

\begin{proof}
Every point $x\in K$ has a $PU(n)$-equivariant neighborhood $U_x$ such that $\rd(U_x,(P^{fr})^c)>0$. Compactness of $K$ also implies that we may find finitely many $x_1,\ldots x_n$ so that $U:=\cup_{j=1}^N U_{x_j}$ is a $PU(n)$-equivariant neighborhood. This neighborhood clearly satisfies $\rd(U,(P^{fr})^c)>0$ and $U\subseteq P^{fr}$.
\end{proof}

\begin{remark}
\label{ddremark}
Principal $PU(n)$-bundles $P\to X$ are classified up to isomorphism by the Cech cohomology group $H^1(X,PU(n))$. The short exact sequence $1\to U(1)\to U(n)\to PU(n)\to 1$ and the isomorphism $H^2(X,U(1))\cong H^3(X,\Z)$ gives rise to a Bockstein mapping $\delta:H^1(X,PU(n))\to H^3(X,\Z)$. Since $Z(SU(n))\cong \Z/n\Z$ this mapping factors over $H^2(X,\Z/n\Z)$ and it follows that $n\delta(P)=0$ for any principal $PU(n)$-bundle $P$. The invariant $\delta(P)$ is known as the Dixmier-Douady invariant of $P$. A theorem of Serre (see \cite{grothen}) states that the torsion in $H^3(X,\Z)$ classifies principal $PU(n)$-bundles up to a $PU(\infty)$-isomorphism acting as the identity on the base $X$. Observe here that we are abusing the notation since a $PU(\infty)$-isomorphism need not preserve the rank $n$.
\end{remark}

\subsection{Projective bundles}

In this subsection, $P$ denotes a locally compact $PU(n)$-space. With the $PU(n)$-action on $P$ and the quotient $SU(n)\to PU(n)$ there is a naturally associated $SU(n)$-action on $P$. For notational simplicity we set $Z_n:=Z(SU(n))$ and identify $\Z/n\Z$ with the Pontryagin dual of $Z_n$. We note that if $E\to P$ is an $SU(n)$-equivariant vector bundle, any $g\in Z_n$ induces an $SU(n)$-equivariant vector bundle automorphism $g_E:E\to E$.

\begin{deef}
\label{projectivebundlemarker}
An $SU(n)$-equivariant vector bundle $E\to P$ is said to have central character $\chi\in \Z/n\Z$ if for any $g\in Z_n$ it holds that
$$g_E=\chi(g)\id_E.$$
A projective bundle over $P$ is an $SU(n)$-equivariant vector bundle $E\to P$ with central character $\chi=1\mod n\Z$.
\end{deef}

We note the following properties of $SU(n)$-equivariant vector bundles with central characters.

\begin{prop}
\label{elementpropproj}
Let $E$, $E'\to P$ be two $SU(n)$-equivariant vector bundles with central characters $\chi$ respectively $\chi'$. It holds that:
\begin{enumerate}
\item If $\chi=\chi'$, the direct sum $E\oplus E'$ is an $SU(n)$-equivariant vector bundle with central character $\chi$.
\item The tensor product $E\otimes_P E'$ is an $SU(n)$-equivariant vector bundle with central character $\chi+\chi'$.
\item The conjugated vector bundle $\overline{E}$ is an $SU(n)$-equivariant vector bundle with central character $-\chi$.
\item If $\chi =0\mod n\Z$, there exists a vector bundle $E_0\to PU(n)\backslash P$ unique up to isomorphism such that $E\cong \pi^*E_0$ as $SU(n)$-equivariant vector bundles (here $\pi:P\to PU(n)\backslash P$ denotes the quotient).
\end{enumerate}
\end{prop}

One can stabilize projective bundles, i.e. when the central character is $1\mod n\Z$, along principal $U(l)$-bundles. This construction is to be compared with Morita equivalence of Azumaya bundles.  Let $Q\to P$ be a $PU(n)$-invariant principal $U(l)$-bundle and $E\to P$ a projective bundle. We set 
$$S(U(n), U(l)):=\{(g,h)\in U(n)\times U(l)| i_{n,l}(g,h)\in SU(nl)\}.$$ 
We let $\pi_{\tilde{Q}}:\tilde{Q}:=Q/U(1)\to P$ denote the principal $PU(l)$-bundle associated with $Q\to P$. The quotient mappings induce a surjection $S(U(n),U(l))\times Z_{nl}\to PU(n)\times PU(l)$ and we equip $\tilde{Q}$ with the associated $S(U(n),U(l))\times Z_{nl}$-action. Since $E$ is a projective bundle, the $SU(n)$-action extends to a $U(n)$-action. As such, we can view the vector bundle $\pi_{\tilde{Q}}^*E\to \tilde{Q}$ as an $S(U(n),U(l))\times Z_{nl}$-equivariant vector bundle, with $Z_{nl}$ acting along the character $1\mod nl\Z$. We define the $PU(nl)$-space
$$P^Q:=PU(nl)\times_{PU(n)\times U(l)}Q=SU(nl)\times_{S(U(n), U(l))\times Z_{nl}}\tilde{Q}$$ 
and the $SU(nl)$-equivariant vector bundle
$$E^Q:=SU(nl)\times_{S(U(n), U(l))\times Z_{nl}}\pi_{\tilde{Q}}^*E\to P^Q.$$

\begin{prop}
The vector bundle $E^Q\to P^Q$ is a well defined projective bundle on the $PU(nl)$-space $P^Q$.
\end{prop}

A morphism of two projective bundles over a $PU(n)$-space is an $SU(n)$-equivariant morphism of vector bundles. We let $Proj(P)$ denote the category of isomorphism classes of projective bundles on $P$. It follows from Proposition \ref{elementpropproj} that $Proj(P)$ is an additive category.

\begin{lem}
\label{zetaeq}
Whenever $Q\to P$ is a $PU(n)$-invariant principal $U(l)$-bundle, stabilization of projective bundles defines an additive equivalence
\[\zeta_Q: Proj(P)\xrightarrow{\sim} Proj(P^Q), \quad \zeta_Q(E):=E^Q.\]
\end{lem}

\begin{proof}
Let $F\to P^Q$ be a projective bundle. There is a $PU(n)\times PU(l)$-equivariant embedding $\tilde{Q}\hookrightarrow P^Q$ and $F|_{\tilde{Q}}$ can be viewed as an $SU(n)\times SU(l)$-equivariant vector bundle with central character 
$$(1\mod n\Z,1\mod l\Z)\in \Z/n\Z\times \Z/l\Z\cong \widehat{Z_n\times Z_l}.$$ 
Define the line bundle $L_Q:=Q\times_{U(1)}\C\to \tilde{Q}$, it forms an $SU(n)\times SU(l)$-equivariant line bundle whose central character is given by $(0\mod n\Z,1\mod l\Z)$. The vector bundle $F|_{\tilde{Q}}\otimes L_Q^{-1}$ forms an $SU(n)\times SU(l)$-equivariant line bundle with central character $(1\mod n\Z,0\mod l\Z)$ (cf. Proposition \ref{elementpropproj}). Since the $PU(l)$-action is free, there is a vector bundle $\zeta^{-1}_Q(F)\to P$ such that $\pi_{\tilde{Q}}^*\zeta^{-1}_Q(F)=F|_{\tilde{Q}}\otimes L_Q^{-1}$. A direct verification of how central elements acts leads us to the conclusion that $\zeta_Q^{-1}:Proj(P^Q)\to Proj(P)$ is the inverse functor to $\zeta_Q$.
\end{proof}

If $f:P\to P'$ is a $PU(n)$-equivariant mapping of $PU(n)$-spaces and $E\to P'$ is a projective bundle then $f^*E$ is clearly a projective bundle on $P$. Projective bundles may also be pulled back along $PU(\infty)$-morphisms: if $f_\infty=(f,Q,Q')$ is a morphism from $P$ to $P'$, we can define its $PU(\infty)$-pullback 
$$f^*_\infty E:=\zeta_{Q'}^{-1}f^*(E^Q)$$ 
which by Lemma \ref{zetaeq} is a well defined isomorphism class of a projective bundle. We let $Vec(X)$ denote the additive category of isomorphism classes of vector bundles on a locally compact Hausdorff space $X$.

\begin{prop}
\label{trivialbundle}
Let $f:P\to PU(n)\times X$ be a $PU(n)$-equivariant homemorphism and $s_f:X\to P$ the section $s_f(x):=f^{-1}(x,1)$, then $s_f^*:Proj(P)\to Vec(X)$ is the inverse functor to $f^*_\infty:Vec(X)\to Proj(P)$ where $f_\infty:=(f,f^*(U(n)\times X),U(n)\times X)$. 
\end{prop}

\begin{proof}
The existence of $f$ guarantees that $P$ forms a principal $PU(n)$-bundle on $X$ that lifts to the $U(n)$-bundle $f^*(U(n)\times X)\to X$. The line bundle 
$$L_f:=f^*(U(n)\times X)\times_{U(1)}\C\to P$$
is a projective line bundle. It is clear from the construction of pull backs that $f_\infty^*E\cong L_f\otimes s_f^*E$ and the Proposition follows.
\end{proof}

\begin{convention} Henceforth, unless stated otherwise, all $PU(\infty)$-morphisms are assumed to be equivariant as all involved operations constructed from a $PU(\infty)$-morphism depend naturally on its components.
\end{convention}

\subsection{Projective $K$-theory}

Before discussing $K$-homology, we recall some notions in $K$-theory. As in Atiyah's original definition of topological $K$-theory, we describe $K$-theory using elliptic complexes. See more in \cite{atiyahbook}.

\begin{deef} 
A projective elliptic complex over a locally compact $PU(n)$-space $P$ is a triple $(E_1,E_2,\sigma)$ where $E_1$ and $E_2$ are projective bundles such that $\sigma:E_1\to E_2$ is a morphism of projective bundles which is an isomorphism outside a compact subset of $P$. 
\end{deef}

A morphism of projective elliptic complexes $(E_1,E_2,\sigma)$ and $(E_1',E_2',\sigma')$ is a pair of morphisms $E_1\to E_1'$ and $E_2\to E_2'$ that intertwines $\sigma$ with $\sigma'$. Composition of morphisms is given by the obvious operation and an isomorphism of projective elliptic complexes is an invertible morphism. A projective elliptic complex of the form $(E,E,\id_E)$ is called an elementary projective elliptic complex. The sum of two projective elliptic complexes $(E_1,E_2,\sigma)$ and $(E_1',E_2',\sigma')$ is defined by
\[(E_1,E_2,\sigma)+(E_1',E_2',\sigma'):=(E_1\oplus E_1',E_2\oplus E_2',\sigma\oplus \sigma').\]
We say that two elliptic complexes $(E_1,E_2,\sigma)$ and $(E_1',E_2',\sigma')$ are degenerately equivalent if there are degenerate cycles $(E,E,\id_E)$ and $(E',E',\id_{E'})$ such that 
$$(E_1,E_2,\sigma)+(E,E,\id_E)\cong (E_1',E_2',\sigma')+(E',E',\id_{E'}).$$

If $f:P\to P'$ is a proper $PU(n)$-equivariant mapping and $(E_1,E_2,\sigma)$ is a projective elliptic complex over $P'$, we define the pulled back projective elliptic complex along $f$ by
\[f^*(E_1,E_2,\sigma):=(f^*E_1,f^*E_2,f^*\sigma).\]
Two projective elliptic complexes $(E_1,E_2,\sigma)$ and $(E_1',E_2',\sigma')$ over $P$ are said to be homotopic if there is a projective complex $(\tilde{E}_1,\tilde{E}_2,\tilde{\sigma})$ over $P\times [0,1]$ such that 
$$i_0^*(\tilde{E}_1,\tilde{E}_2,\tilde{\sigma})\cong(E_1,E_2,\sigma) \quad\mbox{and}\quad i_1^*(\tilde{E}_1,\tilde{E}_2,\tilde{\sigma})\cong(E_1',E_2',\sigma'),$$
where $i_t:P\to P\times [0,1]$ denotes the mapping $i_t(p):=(p,t)$ for $t\in [0,1]$.

\begin{deef}
If $P$ is a principal $PU(n)$-bundle, the projective $K$-theory group of $P$, denoted by $K^{0}_{proj}(P)$, is defined as the group of equivalence classes  of projective elliptic complexes over $P$ under the relation generated by homotopy equivalence and degenerate equivalence. The higher projective $K$-groups are defined by 
$$K^i_{proj}(P):=K^0_{proj}(P\times\R^i).$$
\end{deef}

Indeed, a standard homotopy argument guarantees that for any $j$, $K^{j}_{proj}(P)$ forms an abelian group under addition of projective elliptic complexes. Observe that if $P$ is compact, $K^0_{proj}(P)$ is the Grothendieck group of isomorphism classes of projective bundles. The proof of the next Proposition is the same as in the ordinary setting.

\begin{prop}
Projective $K$-theory is a contravariant functor on the category of principal $PU(n)$-bundles that satisfies Bott periodicity
$K^i_{proj}(P)\cong K^{i+2}_{proj}(P)$. 
\end{prop}

The main reason for introducing projective bundles is the relation with twisted $K$-theory and the module bundles for Azumaya bundles. This relation is explained in Section $1.3$ of \cite{mmstva}. We include it for the sake of completeness. Recall that an Azumaya bundle is a bundle of simple matrix algebras. With a principal $PU(n)$-bundle $P\to X$ there is an associated Azumaya bundle on $X$ defined via 
\begin{equation}
\label{azzzuu}
\mathcal{A}\equiv \mathcal{A}(P):=M_n(\C)\times_{PU(n)} P=M_n(\C)\times_{SU(n)} P,
\end{equation}
where $PU(n)$, respectively $SU(n)$, acts via conjugation on $M_n(\C)$. A vector bundle $E_{\mathcal{A}}\to X$ with a fiberwise action of $\mathcal{A}$ will be called an $\mathcal{A}$-module. If $X$ is a compact Hausdorff space, the Serre-Swan theorem states that there is a one-to-one correspondence between isomorphism classes of $\mathcal{A}$-modules and isomorphism classes of projective finitely generated $C(X,\mathcal{A})$-modules.

If $E_{\mathcal{A}}$ is an $\mathcal{A}$-module, we can choose a cover $(U_\alpha)_{\alpha\in I}$ of $X$ over which $P$ and $E_\mathcal{A}$ trivializes. We use the standard notations $U_{\alpha \beta }:=U_\alpha\cap U_\beta $ and $U_{\alpha \beta \gamma }=U_\alpha\cap U_\beta \cap U_\gamma $. By construction, $E_\mathcal{A}$ is trivializable over $U_\alpha$ and there is a vector bundle $E_\alpha\to U_\alpha$ such that as $\mathcal{A}|_{U_\alpha}$-modules $E_{\mathcal{A}}|_{U_\alpha}\cong E_\alpha\otimes \mathcal{A}|_{U_\alpha}$. Since $E_{\mathcal{A}}\to X$ is a vector bundle there are vector bundle isomorphisms $G_{\alpha \beta }:E_\beta |_{U_{\alpha \beta }}\to E_\alpha |_{U_{\alpha \beta }}$ and a $\C^\times$-valued Cech $2$-cocycle $(\theta_{\alpha \beta \gamma })_{\alpha ,\beta ,\gamma \in I}$ for this covering that satisfy the weak cocycle condition 
\begin{equation}
\label{thetaequation}
\theta_{\alpha \beta \gamma }G_{\alpha \beta }G_{\beta \gamma }=G_{\alpha \gamma }\quad\mbox{on}\quad U_{\alpha \beta \gamma }.
\end{equation}
It is always possible to assume that $(\theta_{\alpha \beta \gamma })_{\alpha ,\beta ,\gamma \in I}$ is $U(1)$-valued. Equation \eqref{thetaequation} can be taken as a definition for a $U(1)$-valued Cech $2$-cocycle $(\theta_{\alpha \beta \gamma })_{\alpha ,\beta ,\gamma \in I}$. It is clear from the construction of the Bockstein mapping that the cohomology class of $(\theta_{\alpha \beta \gamma })_{\alpha ,\beta ,\gamma \in I}$ coincides with $\delta(P)$ under the isomorphism $H^3(X,\Z)\cong H^2(X,U(1))$ (cf. Remark \ref{ddremark}). In \cite{mmstva}, the collection $\left((E_\alpha )_{\alpha \in I},(G_{\alpha \beta })_{\alpha ,\beta \in I}\right)$ is called a projective vector bundle data for the Azumaya bundle $\mathcal{A}$ with respect to the covering $(U_\alpha )_{\alpha \in I}$. After possibly refining $(U_\alpha )_{\alpha \in I}$ we can always assume that $E_\alpha =U_\alpha \times \C^N$ and replace $(G_{\alpha \beta })_{\alpha ,\beta \in I}$ by a collection satisfying Equation \eqref{thetaequation} with $G_{\alpha \beta }:U_{\alpha \beta }\to SU(N)$ for some $N$.

We use the notation 
$$L_n:=SU(n)\times_{Z_n}\C=U(n)\times_{U(1)}\C.$$
The line bundle $L_n\to PU(n)$ is by construction a projective line bundle in the $SU(n)$-action induced from left multiplication on $U(n)$.

\begin{prop}
\label{vbdatatoproj}
Let $\mathcal{A}\to X$ be the Azumaya bundle associated with a principal $PU(n)$-bundle $P\to X$, $(U_\alpha )_{\alpha \in I}$ a fine enough cover for $X$ and $\left((U_\alpha \times \C^N)_{\alpha \in I},(G_{\alpha \beta })_{\alpha ,\beta \in I}\right)$ a projective vector bundle data for the Azumaya bundle $\mathcal{A}$ with respect to the covering $(U_\alpha )_{\alpha \in I}$ with $G_{\alpha \beta }:U_{\alpha \beta }\to SU(N)$. After choosing trivializations $f_\alpha :P|_{U_\alpha }\xrightarrow{\sim} PU(n)\times U_\alpha $, the automorphisms of projective bundles on $PU(n)\times U_{\alpha \beta }$ given by  
$$ \id_{L_n}\boxtimes G_{\alpha \beta }:L_n\boxtimes (U_{\alpha \beta }\times \C^N) \to L_n\boxtimes (U_{\alpha \beta }\times \C^N)$$
lifts to transition functions for a projective vector bundle $E\to P$.
\end{prop}

Here we use $\boxtimes$ to denote exterior tensor product.

\begin{proof}
Since $L_n\to PU(n)$ is a projective line bundle, indeed $L_n\boxtimes (U_{\alpha }\times \C^N)\to PU(n)\times U_\alpha $ form projective vector bundles by Proposition \ref{elementpropproj}. We consider the projective vector bundles 
$$\tilde{E}_\alpha :=f_\alpha ^*(L_n\boxtimes (U_{\alpha }\times \C^N))\to P|_{U_\alpha }.$$ 
Define the morphism of projective vector bundles $\tilde{G}_{\alpha \beta }:\tilde{E}_\beta |_{U_{\alpha \beta }}\to \tilde{E}_\alpha |_{U_{\alpha \beta }}$ from the diagram:
\[\begin{CD}
\tilde{E}_\alpha |_{U_{\alpha \beta }} @>\tilde{G}_{\alpha \beta }>>\tilde{E}_\beta |_{U_{\alpha \beta }}\\
@Af_\alpha ^*AA @Af_\beta ^*AA  \\
L_n\boxtimes (U_{\alpha \beta }\times \C^N)@>\id_{L_n}\boxtimes G_{\alpha \beta }>>L_n\boxtimes (U_{\alpha \beta }\times \C^N) \\
\end{CD}. \]
Using that $(G_{\alpha \beta })_{\alpha ,\beta \in I}$ is an $SU(N)$-valued Cech cochain satisfying Equation \eqref{thetaequation}, the following identity holds in triple intersections:
\begin{align*}
\tilde{G}_{\alpha \beta }\tilde{G}_{\beta \gamma }&=f_\beta ^*(\id_{L_n}\boxtimes G_{\alpha \beta })(f_\alpha ^{-1})^*f_\gamma ^*(\id_{L_n}\boxtimes G_{\beta \gamma })(f_\beta ^{-1})^*=\\
&=f_\gamma ^*(f_\gamma ^{-1})^*f_\beta ^*(\id_{L_n}\boxtimes G_{\alpha \beta })(f_\alpha ^{-1})^*f_\gamma ^*(\id_{L_n}\boxtimes G_{\beta \gamma })(f_\beta ^{-1})^*f_\alpha ^*(f_\alpha ^{-1})^*=\\
&=f_\gamma ^*(\theta_{\alpha \beta \gamma }\boxtimes G_{\alpha \beta }G_{\beta \gamma })(f_\alpha ^{-1})^*=\tilde{G}_{\alpha \gamma }.
\end{align*}
It follows that the morphisms of projective bundles $(\tilde{G}_{\alpha \beta })_{\alpha ,\beta \in I}$ satisfy the cocycle condition and the projective vector bundles $(\tilde{E}_\alpha )_{\alpha \in I}$ can be glued together to a projective vector bundle $\tilde{E}\to P$ using $(\tilde{G}_{\alpha \beta })_{\alpha ,\beta \in I}$. 
\end{proof}

Conversely, $\mathcal{A}(P)$-modules can be constructed from projective bundles on $P$. We leave the proof of the following Proposition to the reader.

\begin{prop}
Let $P\to X$ be a principal $PU(n)$-bundle, $\mathcal{A}:=\mathcal{A}(P)$ the associated Azumaya bundle and $E\to P$ a projective vector bundle. The vector bundle $\C^n\times_{SU(n)} E\to X$ admits a natural $\mathcal{A}$-module structure induced from the final identity of Equation \eqref{azzzuu}.
\end{prop}

\begin{thm}
\label{projectivevstwisted}
If $P\to X$ is a principal $PU(n)$-bundle there is a natural isomorphism $\iota_P:K^*_{proj}(P)\to K_*(C_0(X,\mathcal{A}(P)))$ determined by the property that for compact $X$ and a projective vector bundle $E\to P$:
$$\iota_P[E]=\left[C(X,\C^n\times_{SU(n)}E)\right].$$
\end{thm}

\begin{proof}
Using Bott periodicity, it suffices to consider $*=0$. Let $(E_1,E_2,\sigma)$ be a projective elliptic complex on $P$. We can assume that there is an open pre-compact set $U\subseteq X$ such that $(E_1,E_2,\sigma)$ is degenerate outside $P|_{U}$. As such, there exists a morphism of projective vector bundles $\sigma^{-1}:E_2\to E_1$ such that $\id_{E_1}-\sigma^{-1}\sigma$ and $\id_{E_2}-\sigma\sigma^{-1}$ have compact supports in $U$. We define the isomorphism of projective vector bundles
\[W:= 
\begin{pmatrix} 
(2-\sigma\sigma^{-1})\sigma& \sigma\sigma^{-1}-\id_{E_2}\\
\id_{E_1}-\sigma^{-1}\sigma& \sigma^{-1}
\end{pmatrix}:E_1\oplus E_2\to E_2\oplus E_1.\]

Since $U$ is pre-compact, the Serre-Swan theorem implies that there exist projections $p_{E_1},p_{E_2}\in M_N(C(\overline{U},\mathcal{A}(P|_{\overline{U}})))$ and isomorphisms
$$p_{E_1}C(\overline{U},\mathcal{A}(P)^N)\cong C(\overline{U},\C^n\times _{SU(n)}E_1|_{\overline{U}})\quad\mbox{and}\quad p_{E_2}C(\overline{U},\mathcal{A}(P)^N)\cong C(\overline{U},\C^n\times _{SU(n)}E_2|_{\overline{U}}).$$
Under these isomorphisms, $W$ induces an element $\tilde{W}\in GL_{2N}(C(\overline{U},\mathcal{A}(P|_{\overline{U}})))$ and we can assume that $\tilde{W}(p_{E_1}\oplus 0)\tilde{W}^{-1}=p_{E_2}$ outside a compact subset of $U$. It follows that $[\tilde{W}(p_{E_1}\oplus 0)\tilde{W}^{-1}]-[p_{E_2}]\in K_0(C_c(U,\mathcal{A}(P|_U)))$ is well defined. We let $i_U:K_0(C_c(U,\mathcal{A}(P|_U)))\to K_0(C_c(X,\mathcal{A}(P)))$ denote the mapping associated with the inclusion and set
\[\iota_P(E_1,E_2,\sigma):=i_U\left([\tilde{W}(p_{E_1}\oplus 0)\tilde{W}^{-1}]-[p_{E_2}]\right)\in K_0(C_c(X,\mathcal{A}(P))).\]
By construction, degenerate elliptic complexes are mapped to $0$ and homotopic elements in $K_0(C_c(X,\mathcal{A}(P)))$ are equal. It follows that $\iota_P:K^0_{proj}(P)\to K_0(C_0(X,\mathcal{A}(P)))$ is well defined.

To define an inverse $\iota_P^{-1}:K_0(C_0(X,\mathcal{A}(P)))\to K^0_{proj}(P)$ we note that the dense embedding $C_c(X,\mathcal{A}(P))\to C_0(X,\mathcal{A}(P))$ is isoradial. It follows that the mapping induced by the inclusion $K_0(C_c(X,\mathcal{A}(P)))\to K_0(C_0(X,\mathcal{A}(P)))$ is an isomorphism. Given a class $x\in K_0(C_0(X,\mathcal{A}(P)))$, we can therefore represent $x$ as $x=[p]-[q]$ where $p,q\in M_N(\C1+C_c(X,\mathcal{A}(P)))$ are projections that are Murray-von Neumann equivalent $p\sim q$ outside a compact subset $K$. As such, there exists a $z\in M_{2N}(\C1+C_c(X,\mathcal{A}(P)))$ such that $z(p\oplus 0)=(q\oplus 0)z$ and $z(Z\setminus K)\subseteq GL_{2N}(\C)$. Following Proposition \ref{vbdatatoproj} and the remarks proceeding it there are projective vector bundles $E_p,E_q\to P$ such that there are $C_0(X,\mathcal{A}(P))$-linear isomorphisms
$$pC_0(X,\mathcal{A}(P)^N)\cong C_0(X,\C^n\times _{SU(n)}E_p)\quad\mbox{and}\quad qC_0(X,\mathcal{A}(P)^N)\cong C_0(X,\C^n\times _{SU(n)}E_q).$$
We are now ready to define $\iota_P^{-1}:K_0(C_c(X,\mathcal{A}(P)))\to K^0_{proj}(P)$. We let $\sigma:E_p\to E_q$ denote the morphism of projective vector bundles associated with $z$ and set
\[\iota_P^{-1}(x):=[E_p,E_q,\sigma].\]
The proof that $\iota_P\circ \iota_P^{-1}=\id_{K_0(C_0(X,\mathcal{A}(P)))}$ and $\iota_P^{-1}\circ \iota_P=\id_{K^0_{proj}(P)}$ are straight-forward, but tedious, and left to the reader.
\end{proof}

In this paper, infinite-dimensional Azumaya bundles will also be of interest. They come from principal $PU(\He)$-bundles for a fixed separable infinite-dimensional Hilbert space $\He$. Motivated by Theorem \ref{projectivevstwisted}, we \emph{define} $K^*_{proj}(P)$ as $K_*(C_0(X,\mathcal{A}(P)))$ if $P$ is a $PU(\He)$-bundle and $\He$ is infinite-dimensional and separable.

\begin{prop}
Let $P\to X$ be a $PU(n)$-principal bundle and $Q\to P$ a $PU(n)$-invariant principal $U(l)$-bundle. Let $V_l$ denote the fundamental representation of $U(l)$ and $V_l(Q)\to X$ the associated vector bundle. There is a natural isomorphism
\[\mathcal{A}(P^Q)\cong \mathcal{A}(P)\otimes \mathrm{End}(V_l(Q)).\]
Furthermore, under this isomorphism $\iota_{P^Q}(E^Q)=\iota_P(E)\otimes V_l(Q)$.
\end{prop}

\begin{proof}
The statement follows by observing that
\[\mathcal{A}(P^Q)=M_{nl}(\C)\times_{PU(nl)} P^Q\cong (M_n(\C)\otimes M_l(\C)) \times_{PU(n)\times U(l)}Q\cong \mathcal{A}(P)\otimes \mathrm{End}(V_l(Q)).\]
\end{proof}

A consequence of this Proposition is that the Dixmier-Douady invariant does not depend on the $PU(\infty)$-isomorphism class of $P$. Conversely, the $PU(\infty)$-isomorphism class of a principal $PU(n)$-bundle $P$ is determined by its Dixmier-Douady class, see for instance the proof of \cite[Theorem $9.13$]{cumero}.

\subsection{Cup products of projective bundles}

If $\pi_P:P\to X$ is a principal $PU(n)$-bundle and $\pi_{P'}:P'\to X$ is a principal $PU(n')$-bundle, their fiber product over $X$ is defined by
$$P\times_XP':=\{(p,p')\in P\times P': \;\pi_P(p)=\pi_{P'}(p')\}.$$
The product of $P$ with $P'$ (which is a principal $PU(nn')$-bundle) is defined via
\[PP':=PU(nn')\times_{PU(n)\times PU(n')} (P\times_X P')\to X.\]
It is straightforward to verify that $PP'\cong P'P$ as principal $PU(nn')$-bundles. There is a $PU(n)\times PU(n')$-equivariant mapping $j_{P,P'}:P\times_X P'\to PP'$. If $E$ is a projective bundle over $P$ and $E'$ a projective bundle over $P'$, the tensor product $E\otimes E'\to PP'$ is a projective bundle defined by
$$E\otimes E':=SU(nn')\times_{SU(n)\times SU(n')} (E\boxtimes E').$$
In fact, the projective bundle $E\otimes E'$ is determined by the equation
\[j_{P,P'}^*(E\otimes E')=E\otimes _X E',\quad\mbox{where} \quad E\otimes_X E':=E\boxtimes E'|_{P\times_X P'}.\]

An important special case occurs when $P'=X$ in which case $PP'=P$. If $E\to P$ is a projective vector bundle and $E'\to X$ a vector bundle, then 
$$E\otimes E'=E\otimes \pi_P^*E'=E\otimes_XE'.$$ 
More generally, if $P'$ is trivializable, every projective bundle on $P'$ is of the form $(\pi_{P'})_\infty^*E_0$ for some vector bundle $E_0\to X$ by Proposition \ref{trivialbundle}. Here the pullback is as a stable mapping. In this case, $E\otimes \pi_{P'}^*E_0$ is in fact only the stable pullback of $E\otimes_X E_0$ from $P$ to $PP'$, the same is true if $\delta(P')=0$.

\begin{lem}
\label{productprop}
The tensor product of projective bundles over principal bundles induces a bilinear mapping 
\[K^*_{proj}(P)\times K^*_{proj}(P')\to K^*_{proj}(PP').\]
The tensor product is associative on triple tensor products. Finally, there is a natural isomorphism 
\[\mathcal{A}(PP')\cong \mathcal{A}(P)\otimes_X\mathcal{A}(P')\] 
making the natural transformation $\iota_P$ multiplicative when equipping twisted $K$-theory with the usual cup product.
\end{lem}

\begin{proof}
Since the tensor product is a biadditive and associative operation on the semigroup of projective bundles, the universal property of $K$-groups implies that the same properties holds for projective $K$-theory. That $\iota_P$ is multiplicative under the tensor product follows from that 
\begin{align*}
 M_{nn'}(\C)\times_{PU(nn')}PP'&\cong M_{nn'}(\C)\times_{SU(n)\times SU(n')}(P\times_XP')\cong\\
&\cong(M_n(\C)\times_{SU(n)}P)\otimes_X (M_{n'}(\C)\times_{SU(n')}P')
\end{align*}
and these isomorphisms commute with the construction of $\iota_P$.
\end{proof}

\subsection{Push forwards}
\label{pushforsubsec}

In this subsection, the push forward of a projective bundle will be considered using the map $\iota_P$ and the approach in \cite{thomcareywang}.  Smooth principal $PU(\He)$-bundles over a manifold with boundary are the objects of interest here. To simplify notation we will assume that the manifold $X$ is even-dimensional, this can always be attained after taking cartesian product with $S^1$.

To construct the push forward, we will use the frame bundle $Fr(X)\to X$ of an oriented $2m$-dimensional manifold which is the $SO(2m)$-principal bundle of oriented frames on $T^*X$. The group $Spin(2m)$ acts on the space of complex spinors $S_m$ via the complex spin representation. The space of complex spinors is $2^{m}$-dimensional. The spin-group forms a two-fold cover of $SO(2m)$. Thus we obtain a projective representation $SO(m)\to PU(2^{m})$. Let
\[Pr(X):=Fr(X)\times_{SO(2m)}PU(2^{m})\quad\mbox{and}\quad \tilde{P}:=Pr(X)\cdot P.\]
We use the notation $\C l(X)$ for the complex Clifford algebra bundle over $X$. One has that $\C l(X)=\mathcal{A}(Pr(X))$ and $\delta(Pr(X))=W_3(X)$, the third integral Stiefel-Whitney class. A choice of spin$^c$-structure on $X$ is a lift of $Fr(X)$ to a $Spin^c(m)$-bundle. This in turn produces a vector bundle $S_X\to X$ and an isomorphism $\C l(X)\cong \mathrm{End}(S_X)$.

\begin{prop}
\label{spincstruazu}
There is a natural isomorphism of Azumaya bundles
\[\lambda_P:\mathcal{A}(\tilde{P})\xrightarrow{\sim} \mathcal{A}(P)\otimes \C l(X).\]
In particular, a choice of spin$^c$-structure on $X$ induces an isomorphism
\[\tilde{\iota}_P:K^*_{proj}(\tilde{P})\xrightarrow{\sim} K^*_{proj}(P).\]
\end{prop}

This proposition follows directly from Lemma \ref{productprop} and the fact that the bundle $S_X$ associated with a spin$^c$-structure produces a Morita equivalence $\mathcal{A}(P)\sim_M \mathcal{A}(P)\otimes \C l(X)$. For a manifold with boundary $X$ we use the notation $X^\circ$ for the interior of $X$.

\begin{deef}
\label{push}
Let $P\to X$ and $P'\to X'$ be smooth principal $PU(n)$-bundles on smooth compact manifolds with boundaries. If $f:P\to P'$ is a smooth equivariant mapping such that $f$ restricts to a proper mapping $f|:P|_{X^\circ}\to P'|_{X'^\circ}$, we define the push forward
\[f_!:K^*_{proj}(\tilde{P})\to K^*_{proj}(\tilde{P}')\]
by declaring the following diagram to be commutative
\[\begin{CD}
K^*_{proj}(\tilde{P})@>f_!>>K^*_{proj}(\tilde{P}') \\
@VV\lambda_P\circ \iota_{\tilde{P}}V @VV\lambda_{P'}\circ \iota_{\tilde{P}'}V \\
K_*(C(X,\mathcal{A}(P)\otimes \C l(X)))\;\;@>f_!>> K_*(C(X',\mathcal{A}(P')\otimes \C l(X'))) \\
@VVPD_X V @VVPD_{X'}V \\
K^*(C_0(X^\circ,\mathcal{A}(P)))@>f_*>>K^*(C_0(X'^\circ,\mathcal{A}(P'))) \\
\end{CD},\]

\end{deef}

Our definition of push forward differs from the construction of push-forwards of \cite{thomcareywang} in a very subtle maner. This difference might at first seem trivial but plays a rather important role, for instance for the fractional index. We are considering equivariant mappings $f:P\to P'$ which in regards to $K$-theory contains two equally important pieces of information: the proper mapping on the base $f_0:X\to X'$ and an isomorphism of $PU(\He)$-bundles $\Psi_f:P\cong f^*P'$. The difference between our construction and the one in \cite{thomcareywang} is that they work only with the mapping $f_0$ on the base which indeed induces a push-forward:
\[(f_0)_!:K_*(C(X,\mathcal{A}(f^*P')\otimes \C l(X)))\xrightarrow{\sim}K_*(C(X',\mathcal{A}(P')\otimes \C l(X'))).\]
Using all the data encoded in an equivariant mapping $f$, the push-forward of Definition \ref{push} along $f$ can be recovered from the push-forward of \cite{thomcareywang} along the base by 
\[f_!:=(f_0)_!\Psi_f.\]

The construction of the push-forward in \cite{thomcareywang} uses the fact that the mapping on the base $f_0$ can be factored over the zero section of a vector bundle, an embedding of an open subset and a trivial sphere bundle. Finally, the push-forward is defined on the factors separately. Let us explain what happens if $P'$ is a vector bundle over $P$ and $0:P\to P'$ denotes the zero section. By assumption $P=0^*P'$, $\Psi_0=\id$ and
\[0_!:K^*_{proj}(\tilde{P})\to K^*_{proj}(\tilde{P}'),\]
is conjugate via $\iota_P$  to the twisted Thom isomorphism $K_*(C(X,\mathcal{A}(P)\otimes \C l(X)))\xrightarrow{\sim}K_*(C(X',\mathcal{A}(P')\otimes \C l(X')))$. Furthermore, if $f:P\to P'$ is an embedding of spin$^c$-$PU(n)$-spaces then $f_!:K^*_{proj}(\tilde{P})\to K^*_{proj}(\tilde{P}')$ is the composition of the Thom isomorphism $0_!:K^*_{proj}(\tilde{P})\to K^*_{proj}(N_f)$ onto the horizontal normal bundle of $f$ and extension by $0$. The following Lemma is a direct consequence of the naturality of Poincar\'e duality: 

\begin{lem}
\label{pullbackpushforward}
Push forward commutes with pullback in the sense that if 
\[\begin{CD}
P@>j>>P' \\
@VVqV @VVpV \\
P''@>i>>P''' \\
\end{CD}\]
commutes then the following diagram also commutes:
\[\begin{CD}
K^*_{proj}(P')@>j^*>>K_{proj}^*(P) \\
@VVp_!V @VVq_!V \\
K^*_{proj}(P''')@>i^*>>K^*_{proj}(P'') \\
\end{CD}.\]
\end{lem}

\subsection{Example: Roots of line bundles}
\label{rootsoflinebundles}

As an example of how to construct a principal $PU(n)$-bundle we will use the exponential mapping $H^2(X,\Q)\to H^2(X,\underline{U(1)})$; we let $\underline{U(1)}$ and $\underline{\R}$ denote the constant sheafs and realize the cohomology groups $H^2(X,\Q)$, $H^2(X,\underline{U(1)})$ et cetera as Cech cohomology groups. The Dixmier-Douady class of the principal bundles constructed in this subsection is $0$ as the mapping $H^2(X,\underline{U(1)})\to H^3(X,\Z)$ vanishes on the image of $H^2(X,\underline{\R})\to H^2(X,\underline{U(1)})$. 

Let $\omega\in H^2(X,\Q)$ and assume that we have chosen a Cech cocycle $(\omega_{\alpha\beta\gamma})_{\alpha,\beta,\gamma\in I}$ on a good cover $(U_\alpha)_{\alpha\in I}$ representing it.  Assume that there is an $N\in \N$ such that $N\omega$ is in the image of $H^2(X,\Z)\to H^2(X,\Q)$, this holds for any $\omega$ whenever $\dim_\Q H^2(X,\Q)<\infty$. In particular, there is a cochain $(u_{\alpha\beta})_{\alpha,\beta\in I}\in C^1((U_\alpha)_{\alpha\in I},\R)$ such that $(\e^{iu_{\alpha\beta}})_{\alpha,\beta\in I}\in Z^1((U_\alpha)_{\alpha\in I},U(1))$ and 
\begin{equation}
\label{omegaequation}
\e^{i\omega_{\alpha\beta\gamma}}=\e^{\frac{i}{N}\left(u_{\alpha\beta}+u_{\beta\gamma}u_{\gamma\alpha}\right)}.
\end{equation}
If $N$ is large enough, we can find a cochain $(c_{\alpha\beta})_{\alpha,\beta\in I}\in C^1((U_\alpha)_{\alpha\in I},U(N))$ such that 
\[c_{\alpha\beta}c_{\beta\gamma}c_{\gamma\alpha}=\e^{\frac{i}{N}\left(u_{\alpha\beta}+u_{\beta\gamma}+u_{\gamma\alpha}\right)}.\]
Let us fix such a cochain $(c_{\alpha\beta})_{\alpha,\beta\in I}$ and let $P_\omega$ denote the principal $PU(N)$-bundle associated with the $PU(N)$-cocycle $(c_{\alpha\beta}\mod U(1))_{\alpha,\beta\in I}$. We will also let $L^N_\omega\to X$ denote the line bundle associated with the cocycle $(\e^{iu_{\alpha\beta}})_{\alpha,\beta\in I}$.

\begin{prop}
\label{conel}
If $L_\omega^N$ is the line bundle constructed in the previous paragraph, then $c_1(L_\omega^N)=N\omega$.
\end{prop}

\begin{proof}
This is clear from equation \eqref{omegaequation}.
\end{proof}

\begin{prop}
\label{rootprop}
There is a projective line bundle $L_\omega\to P_\omega$ associated with the cochain $(u_{\alpha\beta})_{\alpha,\beta\in I}\in C^1((U_\alpha)_{\alpha\in I},\R)$ such that
\begin{enumerate}
\item $(L_\omega)^{\otimes N}$ is the lift of $L_\omega^{ N}$ to $P_\omega$.
\item The mapping $K^*(X)\ni E\mapsto E\otimes L_\omega\in K^*_{proj}(P_\omega)$ is an isomorphism.
\end{enumerate}
\end{prop}

\begin{proof}
We define the projective line bundle $L_\omega\to P_\omega$ by gluing together the projective line bundles $U_\alpha\times L_N \to U_\alpha\times PU(N)$ along the $U(N)$-valued $2$-cocycle 
\[\left(c_{\alpha\beta} \exp\left(\frac{i}{N}u_{\alpha\beta}\right):U_{\alpha\beta}\to U(N)\right)_{\alpha,\beta\in I}.\]

That $-\otimes L_\omega:K^*(X)\to K^*_{proj}(P_\omega)$ gives an isomorphism follows from the fact that, if $\tilde{E}\to P_\omega$ is a projective bundle, the $SU(N)$-equivariant vector bundle $\tilde{E}\otimes L_\omega^*$ has central character $0\mod N\Z$ and descends to a vector bundle on $X$ by Proposition \ref{elementpropproj}. 
\end{proof}

\begin{prop}
\label{commdiagramrational}
If $\omega\in H^2(X,\Q)$ then the following diagram commutes:
\[\begin{CD}
K^*(M)@>-\otimes L_\omega>>K^*_{proj}(P_\omega) \\
@V\ch VV @VV\ch_{P_\omega} V \\
H^*_{dR}(X)@>\wedge \e^{\omega}>> H^*_{dR}(X) \\
\end{CD}.\]
\end{prop}

\begin{proof}
By Proposition \ref{rootprop} $(L_\omega)^N=\pi_{P_\omega}^*L_\omega^N$ and by Proposition \ref{conel} $\ch[L_\omega^N]=\e^{N\omega}$. In particular,  
\[\ch_{P_\omega}[L_\omega]=\sqrt[N]{\e^{N\omega}}=\e^\omega.\] 
The proposition follows from the multiplicativity of the Chern character.
\end{proof}

\Large
\section{Projective $K$-homology}
\normalsize
\label{sectionprojectivekhomology}

In this section, we define the projective cycles. They can be seen as an analogue of simplicial homology in twisted $K$-homology. A projective cycle, which is a building block in this homology theory, can be seen as locally capturing the projective structure of the $PU(\He)$-space. 

To define projective cycles we need a generalization of a $PU(\infty)$-morphism (see Definition \ref{puinftymorphism}) to the infinite-dimensional setting. Assume that $\He$, $\He'$ and $\He''$ are separable Hilbert spaces and there is a fixed choice of isomorphism $\He\otimes \He''\cong \He'\otimes \He''$. This isomorphism induces an embedding $PU(\He')\to PU(\He\otimes \He'')$ continuous in the compact-open topology and smooth in the norm topology. Since our purpose is to generalize Definition \ref{puinftymorphism}, it suffices to consider the case that $\He''$ is infinite-dimensional. 

\begin{deef}
Let $\He$, $\He'$ and $\He''$ be as above. A stable morphism $P\to P'$ between a principal $PU(\He)$-bundle $P\to X$ and a principal $PU(\He')$-bundle $P'\to Y$, is a $PU(\He\otimes \He'')$-equivariant mapping
\[f:P\times_{PU(\He)}PU(\He\otimes \He'')\to P'\times_{PU(\He')} PU(\He\otimes \He'').\]
If $f$ is a homeomorphism, we say that $f$ is a stable isomorphism. If $P$, $P'$ are smooth and $f$ is a diffeomorphism we say that $f$ is a stable diffemorphism.
\end{deef}

\begin{remark}
If $P$ is a principal $PU(n)$-bundle and $P'$ is a principal $PU(k')$-bundle any $PU(\infty)$-morphism $(f,P^U,P'^U)$ induces a stable morphism $P\to P'$. This is a consequence of the fact that for any separable infinite-dimensional Hilbert space $\He''$, there is an isomorphism $\C^m\otimes \He''\cong \C^n\otimes \He''$ for any $m,n$ and the $U(\C^l\otimes \He)$-bundle $P^U\times_{U(l)} U(\C^l\otimes \He)\to P$, as well as $P'^U\times_{U(l')} U(\C^{l'}\otimes \He)\to P'$, form trivializable principal $U(\C^l\otimes \He'')$-bundles by Kuiper's theorem.
\end{remark}

We note the following proposition which is a direct consequence of the definition of stable morphism.

\begin{prop}
\label{hompropstab}
A stable morphism $\phi:P\to P'$ of principal $PU(\He)$-, respectively $PU(\He)$-bundles lifts a mapping $\phi_0:P/PU(\He)\to P'/PU(\He')$ such that $\phi_0^*\delta(P')=\delta(P)$ in $H^3(P/PU(\He),\Z)$ and an explicit homotopy, unique up to homotopy, between $\phi_0^*\delta(P')$ and $\delta(P)$ can be constructed from $\phi$.
\end{prop}

\begin{deef}[Projective cycles]
\label{projectivegeometriccycles}
Let $P$ be a principal $PU(\He)$-bundle or a $PU(n)$-space. A projective cycle $(M,E,\phi)$ over $P$ consists of a closed spin$^c$ manifold $M$ with a free smooth spin$^c$-preserving $PU(n)$-action, a projective bundle $E\to M$ and a stable morphism $\phi:M\to P$. A triple of the form $(M,[E_1]-[E_2],\phi)$, with $[E_1]-[E_2]\in K^0_{proj}(M)$ and $(M,E_1,\phi)$ and $(M,E_2,\phi)$ being projective cycles, is called a projective cycle with $K$-theory data.

A triple that satisfies all the conditions on a projective cycle (with $K$-theory data) except that the manifold has a boundary, is called a projective cycle with boundary (and $K$-theory data).
\end{deef}

We sometimes use the terminology projective cycle with bundle data to distinguish a projective cycle from a projective cycle with $K$-theory data. 

\begin{remark}
Let us emphasize the fact that the manifold $M$ in a projective cycle is finite-dimensional. Even though the rank of a $PU(n)$-bundle is irrelevant for our purposes, we shall try to distinguish the notation between the rank of the fixed bundle $P$ and that of the cycles $M$, they will be denoted by $n$ and $k$ respectively. 
\end{remark}

\begin{deef}
An isomorphism of two projective cycles $(M,E,\phi)$ and $(M',E',\phi')$ with bundle data is a pair $(f_0,f_1)$ where $f_0$ is a $PU(\infty)$-diffeomorphism $M\to M'$ (recall Definition \ref{puinftymorphism}) that preserves spin$^c$-structures, $\phi=\phi'\circ f_0$ (composition in the sense of stable morphisms) and $f_1$ is an isomorphism $f_{0,\infty}^*E'\to E$. 

Similarly, an isomorphism of two projective cycles with $K$-theory data $(M,[E_1]-[E_2],\phi)$ and $(M',[E_1']-[E_2'],\phi')$ is a $PU(\infty)$-diffeomorphism $f_0:M\to M'$ such that $f_{0,\infty}^*([E_1']-[E_2'])=[E_1]-[E_2]$ in $K^0_{proj}(M)$. 
\end{deef}

\begin{remark}
Often when referring to a projective cycle (with $K$-theory data) we refer to an isomorphism class of a projective cycle (with $K$-theory data). The projective cycles (with $K$-theory data) have the same natural operations as geometric cycles (see \cite{baumdtva}), e.g. disjoint union. 
\end{remark}

We now turn to the equivalence relation on the projective cycles. This goes along the lines of \cite[Section 11]{baumdtva}. The equivalence relation is generated by three elementary steps: disjoint union/direct sum, bordism and vector bundle modification. We first recall the first two in the context of projective cycles.

\begin{enumerate}
\item The disjoint union/direct sum-relation is a relation on the projective cycles with bundle data identifying $(M,E,\phi)\dot{\cup}(M,E',\phi)$ with $(M,E\oplus E',\phi)$. 
\item Two cycles $(M,E,\phi)$ and $(M',E',\phi')$ are said to be bordant, written 
$$(M,E,\phi)\sim_{bor}(M',E',\phi'),$$ 
if there is a projective cycle with boundary $(W,F,f)$, also referred to as a projective bordism, such that 
$$(M,E,\phi)\dot{\cup}(-M',E',\phi')=(\partial W,F|_{\partial W},f|_{\partial W}).$$
Here $-M'$ denotes $M'$ equipped with its opposite spin$^c$-structure. One defines bordism of projective cycles with $K$-theory data in a similar way. 
\end{enumerate}

\begin{prop}
\label{bordvbvskt}
Both the set of equivalence classes of projective cycles with bundle data, under the relation generated by bordism and disjoint union/direct sum, and the set of equivalence classes of projective cycles with $K$-theory data, under the relation generated by bordism and disjoint union/direct sum, form abelian groups under disjoint union. Furthermore, these abelian groups are isomorphic via the mapping defined on cycles by $(M,E,\phi)\mapsto (M,[E],\phi)$.
\end{prop}

\begin{proof}
Since $(M,E,\phi)\dot{\cup}(-M,E,\phi)=\partial (M\times [0,1],\pi_M^*E,\phi\circ \pi_M)$ the first statements are clear. As for the second statement, it follows from the fact that the mapping  $(M,E,\phi)\mapsto (M,[E],\phi)$ admits an inverse mapping that is defined by $(M,[E_1]-[E_2],\phi)\mapsto (M,E_1,\phi)\dot{\cup}(-M,E_2,\phi)$ -- a mapping that is well defined modulo bordism and the disjoint union/direct sum-relation. 
\end{proof}

The final ingredient in the relation defining projective $K$-homology is vector bundle modification. Recall that the two-out-of-three lemma for spin$^c$-structures implies that if $V_0\to Y$ is an even-dimensional Riemannian spin$^c$-vector bundle over a manifold $Y$ and $Y^{V_0}$ denotes the sphere bundle $S(V_0\oplus 1_\R)$,  a spin$^c$-structure on $Y$ induces one on $Y^{V_0}$. If $V_0$ has a spin$^c$-structure, there is a fiberwise Bott bundle $Q_{V_0}\to Y^{V_0}$ and a Bott class $[\beta_{V_0}]=[Q_V]-\mathrm{rk}\, Q_V[1]\in K^0(Y^{V_0})$, for details see for instance \cite[Chapter 2.5]{Rav}. The case of interest to us is $Y=PU(k)\backslash M$. Let $\pi_M:M\to PU(k)\backslash M$ denote the projection and define $V:=\pi_M^*V_0$. The projection mapping $\pi_{Y^{V_0}}:M^V\to Y^{V_0}$ is a principal $PU(k)$-bundle that respects the spin$^c$-structures. Let $\pi_{M^V}:M^V\to M$ denote the projection mapping.

\begin{deef}[Vector bundle modification]
\label{vebemod}
Let $M$ be a principal $PU(k)$-bundle and $V=\pi_M^*V_0$ where $V_0\to PU(k)\backslash M$ is an even-dimensional spin$^c$-vector bundle. If $(M,E,\phi)$ is a projective cycle with bundle data for $P$, then the projective cycle with bundle data vector bundle modified by $V$ from $(M,E,\phi)$ is 
\[(M,E,\phi)^V:=(M^V,E^V, \phi\circ \pi_{M^V})\quad\mbox{where}\quad E^V:=\pi_{M^V}^*E\otimes\pi_{Y^{V_0}}^*Q_{V_0}.\]
If $(M,[E_1]-[E_2],\phi)$ is a projective cycle with $K$-theory data for $P$, then the projective cycle with $K$-theory data vector bundle modified by $V$ from $(M,[E_1]-[E_2],\phi)$ is 
\[(M,[E_1]-[E_2],\phi)^{[V]}:=\left(M^V,\pi_{M^V}^*\left([E_1]-[E_2]\right)\otimes\pi_{Y^{V_0}}^*[\beta_{V_0}], \phi\circ \pi_{M^V}\right).\]
\end{deef}

\begin{prop}
In the notation of Definition \ref{vebemod}, there is a bordism of projective cycles with $K$-theory data
$$(M^V,[E^V],\phi\circ \pi_{M^V})\sim_{bor} (M,[E],\phi)^{[V]}.$$
\end{prop}

\begin{proof}
Define $M^B:=\bar{B}(V\oplus 1_\R)$, the closed ball bundle, and let $\pi_{M^B}:M^B\to M$ denote the projection. Since $[\beta_{V_0}]=[Q_V]-\mathrm{rk}\, Q_V[1]$, the Proposition follows from that 
$$(M,[E],\phi)^{[V]}=(M^V,[E^V],\phi\circ \pi_{M^V})\dot{\cup} \left(M^V,\mathrm{rk}\, Q_V[\pi_{M^V}^*E], \phi\circ \pi_{M^V}\right)$$
and the latter is the boundary of the projective cycle with boundary 
$$(M^B,\mathrm{rk}\, Q_V[\pi_{M^B}^*E], \phi\circ \pi_{M^B}).$$
\end{proof}

\begin{deef}
\label{theprojectivekhomology}
We define $K_*^{proj}(P)$ as the group of isomorphism class of projective cycles modulo the equivalence relation generated by disjoint union/direct sum, bordism and vector bundle modification. 
\end{deef}

The set $K_*^{proj}(P)$ is in fact an abelian group under disjoint union; the inverse of $(M,E,\phi)$ is by the proof of Proposition \ref{bordvbvskt} given by $-(M,E,\phi):=(-M,E,\phi)$. The group $K_*^{proj}(P)$ is a $\Z/2\Z$-graded group where the degree of a cycle $(M,E,\phi)$ is the dimension modulo $2$ of (the connected components of) $PU(k)\backslash M$.

\begin{remark}
\label{borddefkproj}
It follows from Proposition \ref{bordvbvskt}, and an argument identical to that proving \cite[Proposition $4.3.2$]{Rav}, that $K_*^{proj}(P)$ coincides as an abelian $\Z/2\Z$-graded group with that defined from the projective cycles with $K$-theory data modulo the relation generated by bordism and vector bundle modification. 
\end{remark}

\subsection{The projective analytic assembly mapping}
\label{sectionassemblymapping}

The projective $K$-homology groups are related to more analytically defined groups; twisted $K$-homology groups. For a principal $PU(\He)$-bundle $P\to X$, recall the notation $\mathcal{A}(P):=\Ko(\He)\times_{PU(\He)}P$. We will use the notation
$$K_*^{an}(P):=K^*(C_0(X,\mathcal{A}(P)))=KK_*(C_0(X,\mathcal{A}(P)),\C).$$
The aim of this subsection is to construct an analytic assembly mapping $\mu:K_*^{proj}(P)\to K_*^{an}(P)$ generalizing the classical analytic assembly mapping in geometric $K$-homology, see for instance \cite[Section 18]{baumdtva} or \cite[Definition 3.3]{boosw}.

With a spin$^c$-structure on a closed manifold $Y$ there is an associated fundamental class $[Y]\in K_{\dim Y}^{an}(Y)=K^{\dim Y}(C(Y))$. The fundamental class determines an isomorphism $K^{*}(Y)\xrightarrow{\sim} K_{*+\dim Y}^{an}(Y)$, $x\mapsto [Y]\cap x$. In fact, for any principal $PU(\He)$-bundle $P\to Y$, the fundamental class induces an isomorphism $K_{*}(C(Y,\mathcal{A}(P)))\xrightarrow{\sim} K_{*+\dim Y}^{an}(P)$, $x\mapsto [Y]\cap x$. For details, see \cite{ecemhy}.

Let $(M,E,\phi)$ be a projective cycle for a principal $PU(\He)$-bundle $P\to X$. Since $M$ has a spin$^c$-preserving $PU(k)$-action, we can consider the fundamental class $[PU(k)\backslash M]\in K^*(C(PU(k)\backslash M))$ and define the twisted cap product 
$$[PU(k)\backslash M]\cap\iota_M[E]\in K^*(C(PU(k)\backslash M,\mathcal{A}(M))).$$ 
By definition, the stably equivariant mapping $\phi$ does, up to a stabilization by the endomorphism bundle of a vector bundle specified by $\phi$, satisfy that $\phi^*\mathcal{A}(P)=\mathcal{A}(M)$. 

\begin{deef}
We define the projective analytic assembly of the projective cycle $(M,E,\phi)$ as the class 
\[\mu(M,E,\phi):=\phi_{*}([PU(k)\backslash M]\cap\iota_M[E])\in K_*^{an}(P)=K^*(C_0(X,\mathcal{A}(P))).\] 
\end{deef}

It is clear that $\mu$ is well defined as a mapping from projective cycles to $K_*^{an}(P)$. Observe that if $P\to X$ is smooth and oriented, then by definition $\mu(M,E,\phi)=[X]\cap \phi_{!}[E]$. 

\begin{prop}
The mapping $\mu$ induces a well defined and natural mapping 
\[\mu:K_*^{proj}(P)\to K_*^{an}(P).\]
\end{prop}

\begin{proof}
Following Remark \ref{borddefkproj}, we will prove that $\mu$ induces a well defined mapping on projective cycles with $K$-theory data respecting vector bundle modification and bordism. If we vector bundle modify a projective cycle $(M,[E_1]-[E_2],f)$ with $K$-theory data along $V$, the modified cycle can be realized as $(M^V,s_!([E_1]-[E_2]),f\circ \pi)$ where $s:M\to M^V$ denotes the south pole mapping. But $\pi\circ s=\id_M$ and naturality of push-forwards (see Lemma \ref{pullbackpushforward}) implies that
\begin{align*}
\mu(M^V,s_!([E_1]-[E_2]),f\circ \pi)&=f_{*}\pi_*([M^V/PU(k)]\cap\iota_{M^V}s_!([E_1]-[E_2]))\\
&=f_{*}\pi_*s_*([PU(k)\backslash M]\cap\iota_M([E_1]-[E_2]))\\
&=\mu(M,[E_1]-[E_2],f).
\end{align*}

We let $(W,x,\phi)$ be a projective cycle with boundary and $K$-theory data $x\in K^0_{proj}(W)$. Further, we let $i:\partial W\hookrightarrow W$ denote the embedding and $\partial:K_*(W^o/PU(k))\to K_{*+1}(\partial W/PU(k))$ the boundary mapping. It is well-known that $\partial[W^o/PU(k)]=[\partial W/PU(k)]$, see for instance \cite[ Theorem 4.7]{baumequivalence}. Thus 
\begin{align*}
\mu(\partial W,x|_{\partial W},\phi|_{\partial W}):&=(\phi\circ i)_{*}\left([\partial W/PU(k)]\cap\iota_{\partial W}(x|_{\partial W})\right)=\\
&=\phi_*\circ i_{*}\circ\partial ([W^o/PU(k)]\cap\iota_{W}x)=0,
\end{align*}
as $i_*\circ\partial=0$.
\end{proof}

\begin{thm}
\label{analyticassembly}
If $P$ is a principal $PU(n)$-bundle over a finite $CW$-complex the projective analytic assembly mapping $\mu:K_*^{proj}(P)\to K_*^{an}(P)$ is an isomorphism.
\end{thm}

The proof of Theorem \ref{analyticassembly} is very similar to that of \cite[Theorem 3.11]{boosw}. The similarity comes from the fact that \cite{boosw} deals with equivariant geometric $K$-homology and the geometric properties of the manifold part of the cycles are very similar. The difference in the vector bundles that comes from the extra structure of $SU(k)$-equivariant vector bundles with a specified central character (recall Definition \ref{projectivebundlemarker} on page \pageref{projectivebundlemarker}) does not affect the main idea in the proof as the isomorphism $\iota_P$ of Theorem \ref{projectivevstwisted} allows us to transfer the projective bundles to twisted $K$-theory where one has a similar toolbox with push-forwards and Poincar\'e duality.  The details of the proof will occupy the remainder of this section.

\begin{lem}
\label{retract}
There is a $PU(n)$-invariant retraction $P\xrightarrow{j} N\xrightarrow{p} P$ into a spin$^c$ $PU(n)$-principal bundle $N$.
\end{lem}

\begin{proof}
Based on the proof of \cite[Lemma 2.1]{boosw}, there is a $PU(n)$-equivariant embedding of $P$ into a complex $PU(n)$-representation $W$ and for a small enough $PU(n)$-invariant neighborhood $U$ of $P$ there is a retraction $U\to P$. Furthermore, there is a $V\subseteq U$ that is a $PU(n)$-invariant manifold with boundary. We let $N$ denote the double of $V$, $j:P\to N$ the obvious embedding and $p:N\to P$ the fold map composed with $U\to P$. The spin$^c$ structure on $N$ comes from the complex structure on $W$ so it is compatible with the $PU(n)$-action. Since $P$ is compact, we can assume that $U$ is small enough so that $\rd(U, (W^{fr})^c)>0$ by Proposition \ref{freecompactneighborhood} so the $PU(n)$-action on $U$ is free. Thus $N$ is a principal $PU(n)$-bundle by the Slice Theorem.
\end{proof}

This lemma allows us to define the mapping $\beta:K_*^{an}(P)\to K_*^{proj}(P)$ by
\[\beta(x):=[N,PD_N j_*(x),p].\]

\begin{prop}
The mapping $\beta$ is well defined and natural.
\end{prop}

This is obvious for a fixed retract of $P$. The proof that $\mu$ is an isomorphism will consist of a proof that $\beta$ is its inverse.  From which, it will follow that $\beta$ is independent of the retract.

\begin{lem}
\label{embedding}
If $(M,E,f\circ h)$ is a projective cycle such that $h:M\to N$ is a $PU(n)$-equivariant inclusion of spin$^c$-manifolds with normal bundle $\nu$, the vector bundle modification of $(M,E,f\circ h)$ along $\C\oplus \nu$ is bordant to the vector bundle modification of $(N,h_!E,f)$ along $\C$.
\end{lem}

We will not prove Lemma \ref{embedding} since the proof is the same as that of \cite[Theorem 4.1]{boosw} mutatis mutandis; the same bordism can be used due to Lemma \ref{pullbackpushforward}. 

\begin{lem}
\label{movingoverf}
If $P$ is a principal spin$^c$-$PU(n)$-bundle and $(M,E,f)$ is a projective cycle over $P$,
\[[M,E,f]=[P,f_!E,\id].\]
\end{lem}

The proof is mutatis mutandis the same as the final paragraphs of \cite[Section 4]{boosw} but since there is a small difference in the cycles, we include a sketch of this proof.

\begin{proof}
After stabilization, using the data contained in $\phi$, we can assume that $M$ is a $PU(n)$-bundle. Choose a $PU(n)$-equivariant embedding $j:M\to V$ into a complex $PU(n)$-representation $V$. Let $V^+$ denote the sphere in $V\times \R$ and $\mathrm{pr}_P:P\times V^+\to P$ the projection onto the first coordinate. The mapping $h:=f\times j:M\to P\times V^+$ is a $PU(n)$-equivariant embedding and $\mathrm{pr}_P\circ h=f$. The $PU(n)$-action on $P\times V^+$ is free, because the action on $P$ is free. The embedding $j$ is homotopic to the south pole mapping $c:M\to V^+$ so $h_!=(f\times c)_!$. Lemma \ref{pullbackpushforward} and Lemma \ref{embedding} implies that
\begin{align*}
[P,f_!E,\id_P]&=[P,f_!E,\mathrm{pr}_P\circ (\id_P\times c)]=[P, (\id_P\times c)_!f_!E,\mathrm{pr}_P]=\\
&=[P,h_!E,\mathrm{pr}_P]=[M,E,\mathrm{pr}_P\circ h]=[M,E,f].
\end{align*}
\end{proof}

\begin{proof}[Proof of Theorem \ref{analyticassembly}]
We will prove the claim that $\beta$ is an inverse to $\mu$. By naturality, we can assume that $P$ is a $PU(n)$-spin$^c$-manifold. By Lemma \ref{movingoverf}, we have that 
\[\beta\mu[M,E,f]=\beta([X]\cap f_![E])=[P,f_![E],\id]=[M,E,f].\]
It follows that $\mu$ is injective with right inverse $\beta$. Finally, we observe that Poincar\'e duality implies that $\mu$ is surjective. So $\mu$ is a bijection and the uniqueness of inverses implies that $\beta$ also forms a left inverse to $\mu$.
\end{proof}

Our method of proof that $\mu$ is an isomorphism breaks down whenever the Dixmier-Douady invariant is non-torsion. A similar type of result for non-torsion Dixmier-Douady invariants that concerns the assembly in twisted geometric $K$-homology can be found in \cite{wangcycles}; we discuss it below. In this case, one uses homotopy theoretic methods. Another approach, which can be found in \cite{BCW}, is to allow for general twisted $K$-theory data in the cycles. For the cycles in \cite{BCW}, the proof above goes through in the non-torsion case. In the case of non-torsion twist, it is unclear if the assembly mapping on projective cycles is an isomorphism. Despite this we will now give a class of examples of non-torsion Dixmier-Douady invariant where it is surjective.

\Large
\section{Examples}
\normalsize
\label{sectionexamples}

\subsection{$T$-duality and projective cycles}
\label{tdualprojective}

A very direct way to construct projective cycles is on the $T$-dual of a circle bundle. We will take a $C^*$-algebraic approach to $T$-duality. For a general overview of $T$-duality, see \cite{rosenberg}. Let us describe $T$-duality briefly; if $P\to Z$ is a principal $PU(\He)$-bundle and $\pi:Z\to X$ is a principal circle bundle there is a $T$-dual principal $PU(\He\otimes L^2(S^1))$-bundle $P^T\to Z^T$ on the $T$-dual principal circle bundle $\pi^T:Z^T\to X$ and these satisfy the conditions
\[\pi_*\delta(P)=c_1(Z^T), \;\; \pi^T_*\delta(P^T)=c_1(Z) \;\;\mbox{and}\;\; K_*(C(Z,\mathcal{A}(P)))\cong K_{*+1}(C(Z^T,\mathcal{A}(P^T))).\]
In the $C^*$-algebraic approach, one lifts the fiberwise $\R$-action on the circle bundle $Z$ to an $\R$-action on $C(Z,\mathcal{A}(P))$, see \cite[Lemma 7.5]{rosenberg}. The crossed product $C(Z,\mathcal{A}(P))\rtimes \R$ is a continuous trace algebra and one defines $Z^T$ and $P^T$ by
\[C(Z^T,\mathcal{A}(P^T)):=C(Z,\mathcal{A}(P))\rtimes \R.\]

The reason that $T$-duality shifts $K$-theory is the Connes-Thom isomorphism (see for example \cite{blackadar}). If $A$ is an $\R-C^*$-algebra with $\R$-action $\alpha$ there is an associated short exact sequence of $\R-C^*$-algebras
\[0\to C_0((0,1),A)\to C_0([0,1),A)\to A\to 0,\]
where the $\R$-action on $A$ is trivial and on $C_0([0,1),A)$ the $\R$-action is defined by $s\cdot f(t):=\alpha_{st}(f(t))$. Since $\R$ is nuclear, crossed product by $\R$ is an exact functor. Using that $C_0((0,1),A)\rtimes \R\cong C_0(\R)\otimes A\rtimes \R$ we arrive at the short exact sequence of $C^*$-algebras
\begin{equation}
\label{ctsequence}
0\to C_0(\R)\otimes A\rtimes \R \to C_0([0,1),A)\rtimes \R\to C_0(\R)\otimes A\to 0.
\end{equation}
This produces the Connes-Thom class $\alpha_{CT}\in KK^1(A,A\rtimes \R)$. Similarly, using Takesaki-Takai duality $A\rtimes \R\rtimes \R\cong A\otimes \Ko$ one can construct $\alpha_{CT}^{-1}\in KK^1(A\rtimes \R,A)$. The Connes-Thom isomorphisms on $K$-theory and $K$-homology comes from the following Kasparov products:
\[-\otimes \alpha_{CT}:K_*(A)\xrightarrow{\sim} K_{*+1}(A\rtimes \R) \quad\mbox{and}\quad \alpha_{CT}\otimes -:K^*(A\rtimes \R)\xrightarrow{\sim} K^{*+1}(A).\]
Observe that if $A=C(Z)$ and $Z\to X$ is a circle bundle, the short exact sequence \eqref{ctsequence} is $C(X)$-linear. Especially, the Connes-Thom isomorphisms are $K^*(X)$-linear.

We will construct projective cycles on the $T$-dual of a circle bundle $\pi_Z:Z\to X$ with $P=Z$. The crossed product $C(Z)\rtimes \R$ is a continuous trace algebra with spectrum $Z^T=X\times S^1$. Let us be a bit more precise and from an open subset $U\subseteq X$, a trivialization $g:\pi_Z^{-1}(U)\to U\times S^1$ and a $\theta_0\in S^1$ construct an isomorphism 
\[g^*:\mathcal{A}(P^T)|_{g^{-1}(U\times S^1\setminus\{\theta_0\})}\cong U\times S^1\setminus \{\theta_0\}\times \Ko(L^2(S^1)).\]
The space $C_c(U\times S^1\times \R)$ can (via $g$) be identified with a dense subset of all sections $g^{-1}(U\times S^1)\to \mathcal{A}(P^T)$ vanishing at the boundary. If $\theta\in \R/\Z=S^1$ we define 
\[\He_\theta:=\{f:\R\to \C:\;f(t+1)=\e^{2\pi i\theta}f(t),\; \int_0^1|f|^2\rd t<\infty\}.\]
It is clear that restriction to $[0,1)$ and periodic extension provides an isomorphism $\He_\theta\cong L^2(S^1)$. Now for any $(x,\theta)\in U\times S^1$ we have a mapping 
$$\pi_{(x,\theta)}:C(X\times S^1,\mathcal{A}(P^T))\to \Ko(\He_\theta)$$ 
defined by means of the representation 
\[\pi_{(x,\theta)}:C_c(U\times S^1\times \R)\to  \Ko(\He_\theta),\; \pi_{(x,\theta)}(a)f(t):=\int_\R a(x,t,t-r)f(r)\rd r.\]
This construction shows that a choice of a trivializations of $Z$ on a cover of $X$ produces trivializations of $P^T$ on the same cover of $X$.

\begin{prop}
\label{tildephi}
Assume that $\phi:M\to Z$ is a continuous mapping to the total space of a circle bundle $\pi_Z:Z\to X$. We let $\tilde{\phi}:M\times S^1\to Z$ denote the equivariant mapping $\tilde{\phi}(m,\theta):=\theta\cdot\phi(m)$. The crossed product construction produces a $PU(L^2(S^1))$-equivariant mapping $\tilde{\phi}\rtimes \R:M\times S^1\times PU(L^2(S^1))\to P^T$ which fits into the commutative diagram
\[\begin{CD}
M\times S^1\times PU(\He)@>\tilde{\phi}\rtimes \R>>P^T \\
@VVV @VVV \\
M\times S^1@>\pi_Z\phi\times \id>>X\times S^1 \\
\end{CD}\]
\end{prop}

\begin{proof}
As $\tilde{\phi}$ is $\R$-equivariant it induces a $*$-homomorphism
\[\tilde{\phi}^*\rtimes \R:C(Z)\times \R\to C(M\times S^1)\rtimes \R=C(M\times S^1,\Ko(L^2(S^1))).\]
This gives the $PU(L^2(S^1))$-equivariant mapping $\tilde{\phi}\rtimes \R:M\times S^1\times PU(L^2(S^1))\to P^T$. It remains to prove that the mapping on the base is $\pi_Z\phi\times \id$. If we identify $(m,\theta)\in M\times S^1$ with the representative $\pi_{(m,\theta)}$ for the class of $(m,\theta)$ in the spectrum of $C(M\times S^1)\rtimes \R$, we must show that, up to a unitary equivalence, 
\[(\tilde{\phi}^*\rtimes \R)\pi_{(m,\theta)}=\pi_{(\pi_Z(\phi(m)),\theta)},\]
where the right hand side is defined using some choice of trivialization for $Z$ around $x$. The last equation now follows from a straight-forward calculation using this trivialization.
\end{proof}

\begin{deef}[The projective $T$-dual of geometric cycle on circle bundle]
If $(M,E,\phi)$ is a geometric cycle on the total space of a circle bundle $Z\to X$ we define its projective $T$-dual cycle $(M,E,\phi)\rtimes \R$ on $P^T$ by
\[(M,E,\phi)\rtimes \R:=(M\times S^1,E,\tilde{\phi}\rtimes \R),\]
where we identify $E$ with its pullback to $M\times S^1$ and $\tilde{\phi}\rtimes \R:M\times S^1\times PU(L^2(S^1))\to P^T$ denotes the mapping of Proposition \ref{tildephi}.
\end{deef}

\begin{prop}
\label{tdualprop}
For a circle bundle $Z\to X$, $T$-duality produces a well defined monomorphism $\rtimes \R:K_*^{geo}(Z)\to K_*^{proj}(P^T)$ that commutes with assembly mappings. That is, the diagram
\[\begin{CD}
K_*^{geo}(Z)@>\rtimes\R>>K_{*+1}^{proj}(P^T) \\
@V\mu^Z VV @VV\mu^{P^T} V \\
K_*^{an}(Z)@>\alpha_{CT}^{-1}>> K_{*+1}^{an}(P^T) \\
\end{CD}\]
commutes. In particular, for flat $Z$ this mapping is an isomorphism.
\end{prop}

\begin{proof}
That $\rtimes \R$ is well defined is clear as it maps disjoint union to disjoint union, a vector bundle modification to a vector bundle modification and a bordism to a bordism. 

There is an isomorphism $K^i(M\times S^1)\cong K^0(M)\oplus K^1(M)$ for $i=0,1$. Under these isomorphisms the Connes-Thom isomorphism is conjugate to a natural $K^*(M)$-linear automorphism of $K^*(M)\otimes \Z^2$. A consequence of this is that if $E$ is the pullback of a vector bundle on $M$, $[E]\otimes \alpha_{CT}=i_![E]$ where $i:M\to M\times S^1$ denotes embedding in the first factor. This fact follows as naturality reduces the calculation to computing the Connes-Thom isomorphism for $S^1$. 

We are now ready to compare $\mu^{P^T}\circ \rtimes \R$ with $\alpha_{CT}^{-1}\circ \mu^Z$. If $(M,E,\phi)$ is a geometric cycle on $Z$, the consideration in the previous paragraph implies 
\begin{align*}
\mu^{P^T}((M,E,\phi) \rtimes \R)&=(\tilde{\phi}\rtimes \R)_*([M\times S^1]\cap[E])=(\tilde{\phi}\rtimes \R)_*\circ \alpha_{CT}^{-1}([M\times S^1]\cap i_![E])=\\
&=(\tilde{\phi}\rtimes \R)_*\circ \alpha_{CT}^{-1}\circ i_*([M]\cap[E])=\alpha_{CT}^{-1}\circ\tilde{\phi}_*\circ  i_*([M]\cap[E])=\\
&=\alpha_{CT}^{-1}\circ\phi_*([M]\cap[E])=\alpha_{CT}^{-1}\circ\mu^Z([M]\cap[E]).
\end{align*}
Here we use the naturality of the Connes-Thom mapping in the second row.
\end{proof}

\begin{cor}
\label{tdualsur}
If $Z\to X$ is a circle bundle with $T$-dual $PU(L^2(S^1))$-bundle $P^T\to X\times S^1$, the projective assembly mapping $\mu:K_*^{proj}(P^T)\to K_*^{an}(P^T)$ is a surjection.
\end{cor}

When $X$ is smooth, this is an instance of a more general result, see Corollary \ref{assemblycorollary}.

\subsection{Clutching construction for projective pseudo-differential operators}
\label{clutching}

The main motivation for introducing projective cycles is to obtain a geometric framework for dealing with the fractional indices of projective pseudo-differential operators. We will return to the subject of fractional indices later in the paper. In this subsection, we will very briefly recall projective pseudo-differential operators and show how an elliptic ditto fits into a projective cycle. The projective pseudo-differential operators can only be defined for a principal $PU(n)$-bundle. The reason is that if the twist is not torsion, there is an obstruction to deforming the associated Azumaya bundle (for more details see \cite{deformationgerbes}). 

To introduce projective pseudo-differential operators over $P$ one needs an extension of $\tilde{\mathcal{A}}(P)$ to a neighborhood of the diagonal $\Delta\subseteq X\times X$ equipped with a multiplication mapping 
\begin{equation}
\label{comp}
\tilde{\mathcal{A}}(P)_{(x,y)}\times\tilde{\mathcal{A}}(P)_{(y,z)}\to \tilde{\mathcal{A}}(P)_{(x,z)},
\end{equation}
for any $x,y,z$ such that all these pair of points lies in this neighborhood defining an associative product whenever it is well defined. Such an extension can be constructed if $P$ is a principal $PU(n)$-bundle, see \cite[Proposition 4]{mms}. For projective bundles $E,F\to P$ the $\Hom$-bundle $\Hom(E,F)\to P$ has central character $0$. Thus, there is an $\mathcal{A}$-linear $\Hom$-bundle $\Hom_\mathcal{A}(E,F)\to X$ that in the same way as above extends to a bundle $\widetilde{\Hom}_{\tilde{\mathcal{A}}}(E,F)$ on a neighborhood of the diagonal in $X\times X$ compatible with the product \eqref{comp}.

A projective pseudo-differential operator is a distribution with values in the extension $\widetilde{\Hom}_{\tilde{\mathcal{A}}}(E,F)$ with a polyhomogeneous expansion near the diagonal and conormal to the diagonal. Composition of two projective pseudo-differential operators is defined using \eqref{comp} if the distributions are supported in small enough neighborhoods of the diagonal. The space of classical projective pseudo-differential operators between two projective bundles $E_1$ and $E_2$ over $P$ is denoted by $\Psi^*(P;E_1,E_2)$, we will in the notation suppress the dependence on the neighborhood and always assume that compositions are defined. 

Let $\Omega\to X$ denote the bundle of densities on $X$ and $\pi_2:X\times X\to X$ projection onto the second coordinate. A smoothing projective operator $T\in\Psi^{-\infty}(P;E_1,E_2)$ has a Schwartz kernel $k_T\in C^\infty_c(\Delta_\epsilon, \widetilde{\Hom}_{\tilde{\mathcal{A}}}(E,F)\otimes \pi_2^*\Omega)$ for an open neighborhood $\Delta_\epsilon$ of the diagonal $\Delta$. If $E=F$, the fiberwise trace induces a bundle mapping $\tra_E:\Hom_\mathcal{A}(E,E)\to 1_\C$, here $1_\C$ denotes the trivial complex line bundle. For $k_T\in C^\infty_c(\Delta_\epsilon, \widetilde{\Hom}_{\tilde{\mathcal{A}}}(E,E)\otimes \pi_2^*\Omega)$, $k_T|_\Delta\in C^\infty(X, \Hom_{\mathcal{A}}(E,E)\otimes \Omega)$ and we arrive at a well defined section $\tra_E(k_T|_\Delta)\in C^\infty(X,\Omega)$. Following \cite[Equation (27)]{mms}, we define the trace functional 
\begin{equation}
\label{tracedefinition}
\tra:\Psi^{-\infty}(P;E,E)\to \C, \quad T\mapsto \int _X\tra_{E}(k_T|_\Delta).
\end{equation}
The nomenclature trace functional derides from that Fubini's theorem implies the identity $\tra \,T_1T_2=\tra \,T_2T_1$ for $T_1\in \Psi^{-\infty}(P;E_1,E_2)$ and $T_2\in \Psi^{-\infty}(P;E_2,E_1)$.

An example of a projective pseudo-differential operator is the projective Dirac operator of an oriented manifold $X$. This is a projective operator over $Pr(X):=Fr(X)\times_{SO(n)}PU(2^m)$-the projective frames on $X$. See more in \cite[Section 7]{mms}.

Let $T^*_\perp P$ denote the pullback of the cotangent bundle $T^*X\to X$ to $P$ and $S^*_\perp P$ denote the pullback of the cosphere bundle $S^*X\to X$ up to $P$. The $PU(n)$-action on $X$ induces the structure of a principal $PU(n)$-bundle on  $S^*_\perp P\to S^*X$. From the definition of a projective pseudo-differential operator (see \cite[Equation (29)]{mms}) it follows that with any $A\in \Psi^m(P;E_1,E_2)$, there is a canonically associated principal symbol 
$$\sigma_m(A)\in C^\infty(S^*_\perp P,\pi_\perp^*\Hom(E_1,E_2))^{SU(n)},$$
where $\pi_\perp:S^*_\perp P\to P$ denotes the projection. A projective pseudo-differential operator $A$ of order $m$ is elliptic if and only if $\sigma_m(A)\in C^\infty(S^*_\perp P,\mathrm{Iso}(\pi_\perp^*E_1,\pi_\perp^*E_2))^{SU(n)}$. If $A$ is elliptic, we can perform a clutching construction to associate a projective cycle with $A$.

Let $\pi:P\to X$ and $\pi_1:\bar{B}^*_\perp P\to P$ denote the projection mappings, here $\bar{B}^*_\perp P$ denotes the closed ball bundle in $T^*_\perp P$. Observe that $T^*P\cong \pi^*T^*X\oplus \mathfrak{pu}(n)^*$, where $\mathfrak{pu}(n)$ denotes the Lie algebra of $PU(n)$, since the $PU(n)$-action on $P$ is free. We have that 
\[T^*\bar{B}^*_\perp P\oplus \mathfrak{pu}(n)\cong T^*(T^*P)|_{\bar{B}^*_\perp P},\]
and this isomorphism produces a $PU(n)$-invariant spin$^c$-structure, coming from the stably almost complex structure on $T^*P$ and $\mathfrak{pu}(n)^*$. 

\begin{deef}
The clutching construction of an elliptic classical projective pseudo-differential operator $A$ is defined as the projective cycle on $P$ given by
\[cl(A):=(S(T^*_\perp P\oplus 1),E_A,\phi),\]
where:
\begin{enumerate}
\item We identify $S(T^*_\perp P\oplus 1)$ with two copies of $\bar{B}^*_\perp P$ glued together along $S^*_\perp P$, and give $S(T^*_\perp P\oplus 1)$ the spin$^c$-structure from this identification. 
\item We define the projective bundle $E_A\to S(T^*_\perp P\oplus 1)$ by gluing together the projective bundle $\pi^*_1E_1\to \bar{B}^*_\perp P$ with the projective bundle $\pi^*_1E_2\to \bar{B}^*_\perp P$ along $S^*_\perp P$ using the $SU(n)$-equivariant isomorphism $\sigma(A):\pi_\perp^*E_1\to \pi_\perp^*E_2$. 
\item We let $\phi:S(T^*_\perp P\oplus 1)\to P$ denote the projection mapping. 
\end{enumerate}
\end{deef}

\begin{prop}
Assume that $P$ is a principal $PU(n)$-bundle over a smooth manifold. The clutching construction produces a well defined mapping
\[c:K^*_{proj}(T_\perp^*P)\to K_*^{proj}(P).\]
\end{prop}

\begin{proof}
It is clear that the clutching construction respects isomorphisms. Furthermore, an elementary projective elliptic complex $(E,E,\id)$ is mapped to the boundary $\partial(B(T^*_\perp P\oplus 1),\phi^*E,\phi)$, where we abuse notation and let $\phi:B(T^*_\perp P\oplus 1)\to P$ denote the projection mapping. 
\end{proof}

\Large
\section{Twisted geometric $K$-homology}
\normalsize
\label{sectiongeometricmodification}

There is also a relationship between projective and twisted geometric $K$-homology. Twisted geometric cycles are defined in the more general setting of a principal $PU(\He)$-bundle $P\to X$ for a separable Hilbert space $\He$. We recall the definition from \cite{wangcycles} of a twisted geometric cycle. We use the notation $BSO(\infty):=\varinjlim BSO(k)$ and $BPU(\He)$ for a classifying space for $PU(\He)$-actions. For topological spaces $X$ and $Y$, we let $[X,Y]$ denote the set of homotopy classes of continuous mappings $X\to Y$. We also let $W_3:BSO(\infty)\to BPU(\He)$ denote the third integral Stiefel-Whitney class\footnote{More concretely, if $V\to X$ is an oriented Riemannian vector bundle, classified by $\nu:X\to BSO(\infty)$ the homotopy class $[W_3\circ \nu]\in [X,BPU(\He)]=H^3(X,\Z)$ is the third integral Stiefel-Whitney class of $V$.}.

\begin{deef}[Definition 3.1 and 6.1, \cite{wangcycles}]
A twisted geometric cycle is a quintuple $(M_0,E_0,\phi_0,\nu,\eta)$, where 
\begin{enumerate}
\item $M_0$ is a closed manifold with stabilized normal bundle $\nu:M_0\to BSO(\infty)$; 
\item $E_0\to M_0$ is a vector bundle; 
\item $\phi_0:M_0\to X$ is a continuous mapping such that 
\[\phi_0^*\delta(P)+W_3(\nu)=0\in H^3(M_0,\Z)=[M_0,BPU(\He)]\]
through an explicit homotopy $\eta:[0,1]\times M_0\to BPU(\He)$. 
\end{enumerate}
\end{deef}

The twisted geometric $K$-homology $K_*^{geo}(P)$ of a principal $PU(\He)$-bundle $P$ is defined as the set of isomorphism classes of twisted geometric cycles modulo direct sum/disjoint union, vector bundle modification and bordism. For details, see \cite[Section 6]{wangcycles}. The set $K_*^{geo}(P)$ forms an abelian group under disjoint union and is $\Z/2\Z$-graded by the dimension of the manifold in the cycle modulo $2$. 

\begin{remark}
Also for twisted geometric cycles, it is possible to work with cycles with $K$-theory data. In the same way as in Section \ref{sectionprojectivekhomology}, the model using bundle data is equivalent to the that defined using $K$-theory data modulo bordisms.
\end{remark}

\begin{remark}
Recall the construction of $Pr(M_0)$ from an orientation in Subsection \ref{pushforsubsec}. A Morita equivalent representative for $Pr(M_0)^{op}$ is constructed from $\nu$; we construct the principal bundle $Pr(\nu)$ as follows.  Representing $\nu$ as a mapping $M_0\to BSO(k)$ with $k$ even, we obtain an $SO(k)$-bundle $Fr(\nu)\to M_0$ of frames on the stable normal bundle. The principal bundle $Pr(\nu)\to M_0$ is the fiber product of the $SO(k)$-bundle $Fr(\nu)\to M_0$ over its projective spin representation with $PU(2^k)$. The stabilization of $Pr(\nu)$ is denoted by $Pr_s(\nu)$ (i.e., $Pr_s(\nu):=Pr(\nu)\times_{PU(2^k)}PU(\He)$).
\end{remark}

\begin{prop}
The homotopy $\eta$ induces a $PU(\He)$-equivariant mapping 
\[\hat{\eta}^0:Pr_s(\nu)\to \phi_0^*P,\]
which in turn induces a stable morphism $\hat{\eta}:Pr(\nu)\to P$. 
\end{prop}

\begin{proof}
As $\eta$ is a homotopy from $\phi_0^*\delta(P)$ to $W_3(\nu)$ it induces a $PU(\He)$-equivariant mapping $\hat{\eta}^0:Pr_s(\nu)\to \phi_0^*P$, see \cite[Section 3]{wangcycles}. The stable equivariant morphism $\hat{\eta}:Pr(\nu)\to P$ is defined by means of the composition 
\[Pr_s(\nu)\xrightarrow{\hat{\eta}^0} \phi_0^*P\to P\] 
where the latter is projection onto the first coordinate when writing $\phi_0^*P=P\times_X M_0$.
\end{proof}

For a closed oriented manifold $M_0$, we let 
$$[M_0]\in K^{\dim M_0}(C(M_0,\C\ell(M_0)))=K_{\dim M_0}^{an}(Pr(\nu))$$ 
denote fundamental class. The geometric assembly mapping $\mu_{geo}:K_*^{geo}(P)\to K_*^{an}(P)$ is defined by
\[\mu_{geo}(M_0,E_0,\phi_0,\nu,\eta)=\hat{\eta}_*([M_0]\cap[E_0]).\]
By \cite[Theorem 6.4]{wangcycles}, the geometric assembly mapping is an isomorphism for any principal $PU(\He)$-bundle over a smooth manifold $X$. The requirement that $X$ is smooth is crucial for the construction of geometric cycles in the proof of \cite[Theorem 6.4]{wangcycles}. 

\subsection{The geometric modification mapping}

We now turn to comparing twisted geometric $K$-homology with projective $K$-homology. We construct a natural mapping $\delta:K_*^{geo}(P)\to K_*^{proj}(P)$ that we call the geometric modification mapping. Recall the notation $\iota_P$ from Theorem \ref{projectivevstwisted} and $\lambda_P$ from Proposition \ref{spincstruazu}.

\begin{deef}
\label{geomoddefe}
Let $(M_0,E_0,\phi_0,\nu,\eta)$ be an even-dimensional twisted geometric $K$-homology cycle, its geometric modification is the projective cycle $(M_\nu,[E_\nu],\hat{\eta}\circ \pi_\nu)$ with $K$-theory data defined as follows:
\begin{enumerate}
\item We can represent $\nu$ by a mapping $\nu:M_0\to BSO(k)$ (we abuse the notation by reusing the letter $\nu$). Pull back $\nu$ to a mapping $\tilde{\nu}:S(T^*M_0\oplus 1_\R)\to BSO(k)$ and set
\[M_\nu:=Pr(\tilde{\nu})=\pi^*Pr(\nu),\]
where $\pi:S(T^*M_0\oplus 1_\R)\to M_0$ denotes the projection mapping and let $\pi_\nu:M_\nu\to Pr(\nu)$ denote the projection.
\item Equipp the total space $M_\nu$ of the $PU(2^k)$-principal bundle $M_\nu\to S(T^*M_0\oplus 1_\R)$ with the spin$^c$-structure induced from $S(T^*M_0\oplus 1_\R)$.
\item Let $s:M_0\to S(T^*M_0\oplus 1)$ denote the south pole map\footnote{By construction, $s^*M_\nu=Pr(\nu)$.} and define $[E_\nu]\in K^0_{proj}(M_\nu)$ as the image of $[E]\in K^0(M_0)$ under the mappings
\begin{align*}
K^0(M_0)\xrightarrow{s_!} &K_0(C(S(T^*M_0\oplus 1),\mathcal{A}(\tilde{M}_\nu)))\\
&\xrightarrow{(\lambda_P)_*} K_0(C(S(T^*M_0\oplus 1),\mathcal{A}(M_\nu)))\xrightarrow{\iota_{M_\nu}} K^0_{proj}(M_\nu).
\end{align*}
where the push forward $s_!$ as well as the mapping $\lambda_P$ is defined using the spin$^c$-structure on $S(T^*M_0\oplus 1)$ and the stabilized normal bundle $\nu$ (see \cite[Theorem 4.1]{thomcareywang}).
\end{enumerate}

If $(M_0,E_0,\phi_0,\nu,\eta)$ is an odd-dimensional twisted geometric $K$-homology cycle, its geometric modification is the projective cycle $(M_\nu,[E_\nu],\hat{\eta}\circ \pi_\nu)$ with $K$-theory data with 
\begin{enumerate}
\item $M_\nu:= Pr(\tilde{\nu})\times S^1\times S^1$ is equipped with the spin$^c$-structure defined analogously to above;
\item $\pi_\nu:M_\nu\to Pr(\nu)$ denotes the projection;
\item $[E_\nu]\in K^0_{proj}(M_\nu)$ is defined, using the mapping 
$$s\times\id_{S^1}:M_0\times S^1\to S(T^*M_0\oplus 1)\times S^1,$$ 
as the image of $[E_0\times S^1]\in K^0(M_0\times S^1)$ under the mappings
\begin{align*}
K^0(M_0\times S^1)&\xrightarrow{(s\times \id_{S^1})_!} K_1(C(S(T^*M_0\oplus 1)\times S^1,\mathcal{A}(\widetilde{Pr}(\tilde{\nu}))))\\
&\to K_0(C(S(T^*M_0\oplus 1)\times S^1\times S^1,\mathcal{A}(\tilde{M}_\nu)))\\
&\xrightarrow{(\lambda_{M_\nu})_*} K_0(C(S(T^*M_0\oplus 1)\times S^1\times S^1,\mathcal{A}(M_\nu)))\xrightarrow{\iota_{M_\nu}} K^0_{proj}(M_\nu).
\end{align*}
where the second mapping is defined from Bott periodicity and the other mappings are defined as in the even-dimensional case.
\end{enumerate}
\end{deef}

The notation $[E_\nu]$ for the $K$-theory class constructed in Definition \ref{geomoddefe} is only figurative, in general there is no projective vector bundle $E_\nu\to M_\nu$ representing $[E_\nu]$.

\begin{thm}
\label{geometricmodification}
If $P\to X$ is a principal $PU(\He)$-bundle, the geometric modification mapping induces a well defined mapping $\delta:K^{geo}_*(P)\to K_*^{proj}(P)$ that fits into the commutative diagram
\begin{equation}
\label{geoprojan}
\xymatrix@C=2.4em@R=3.71em{
& K_*^{geo}(P)\ar[dr]^{\mu_{geo}}\ar[dl]_{\delta}&\\
K_*^{proj}(P)\ar[rr]^{\mu}  && K_*^{an}(P)}.
\end{equation}
\end{thm}

\begin{proof}
Let us start by proving that $\delta$ respects the relations in $K_*^{geo}(P)$. It is clear that $\delta$ respects direct sum/disjoint union. If $(W_0,E_0,\phi_0,\nu,\eta)$ is a twisted geometric cycle with boundary, it is clear that we can construct a projective cycle with boundary $(W_\nu, [E_\nu], \hat{\eta}\circ \pi_\nu)$ in the same manner as above and $\delta\partial (W_0,E_0,\phi_0,\nu,\eta)=\partial(W_\nu, [E_\nu], \hat{\eta}\circ \pi_{\nu})$ is null bordant. 

Similarly, it is standard to show that if $V\to M_0$ is an even-dimensional spin$^c$-vector bundle and letting $\pi_V:M^V:=S(V\oplus 1)\to M_0$, $\tilde{\pi}:M_\nu\to M_0$ and $\tilde{\pi}_V:(M_\nu)^{\pi^* V}:=S(\pi^*V\oplus 1)\to M_\nu$ denote the projections, there is a bordism
\begin{align*}
\delta(M^V,E_0\otimes Q_V,\phi_0\circ \pi_V,\nu_V,\eta_V)&\sim_{bor} ((M_\nu)^{\pi^*V}, [E_\nu]\otimes [\beta_{\pi^*V}], \hat{\eta}\circ\pi_{\nu}\circ \tilde{\pi}_V)\\
&=(M_\nu,[E_\nu],\hat{\eta}\circ \pi_\nu)^{[\tilde{\pi}^*V]}.
\end{align*}

Now we turn to proving that the diagram commutes. If $(M_0,E_0,\phi_0,\nu,\eta)$ is an even-dimensional twisted geometric cycle, we have that 
\begin{align*}
\mu\circ\delta[M_0,E_0,\phi_0,\nu,\eta]&=(\hat{\eta}\circ \pi_{\nu})_{*}\left([S(T^*M_0\oplus 1)]\cap \left(\iota_{M_\nu}\circ \lambda_P\circ s_![E_0]\right)\right)=\\
&=\hat{\eta}_{*}([M_0]\cap [E_0])=\mu_{geo}[M_0,E_0,\phi_0,\nu,\eta].
\end{align*}

The proof for odd-dimensional twisted geometric cycles $(M_0,E_0,\phi_0,\nu,\eta)$ is more complicated. We denote the mapping
\begin{align*}
K_1(C(S(T^*M_0\oplus 1)\times S^1,\mathcal{A}(\widetilde{Pr}(\tilde{\nu}))))\xrightarrow{\sim} K_0(C(S(T^*M_0\oplus 1)\times S^1\times \R,\mathcal{A}(\tilde{M}_\nu|)))\\
\subseteq K_0(C(S(T^*M_0\oplus 1)\times S^1\times S^1,\mathcal{A}(\tilde{M}_\nu))),
\end{align*}
defined using Bott periodicity (see Definition \ref{geomoddefe}) by $b$. We let $H\to S^1\times S^1$ denote the Hopf bundle; the two facts we use about the Hopf bundle are:
\begin{enumerate}
\item The class $[H]-[1_\C]$ is the image of the generator of $K^1(S^1)$ under the mapping $K^1(S^1)\cong K^0(S^1\times \R)\subseteq K^0(S^1\times S^1)$. 
\item Let $\bar{D}$ denote the closed unit disk and identify $S^1=\partial D$. There is a small open ball $B_\epsilon \subseteq S^1\times \bar{D}$ and a line bundle $\bar{H}\to (S^1\times \bar{D})\setminus B_\epsilon$, defined by gluing together the trivial line bundle $\C\times [0,1]\times \bar{D}$ in a suitable fashion, such that $\bar{H}|_{S^1\times S^1}=H$ and $\bar{H}_{\partial B_\epsilon}$ is the Bott bundle on $S^2=\partial B_\epsilon$. 
\end{enumerate} 

Let $\pi_M:M\times S^1\times S^1\to M$ and $\pi_{S^1\times S^1}:M\times S^1\times S^1\to S^1\times S^1$ denote the projections. We arrive at the identities, 
\begin{align*}
\mu\circ\delta&[M_0,E_0,\phi_0,\nu,\eta]\\
&=(\hat{\eta}\circ \pi_{\nu})_{*}\left([S(T^*M_0\oplus 1)\times S^1\times S^1]\cap \left(\iota_{M_\nu}\circ \lambda_P\circ b\circ (s\times \id_{S^1})_![E_0\times S^1]\right)\right)\\
&=\hat{\eta}_{*}\left([M_0\times S^1\times S^1]\cap \left([\pi_M^*E_0\otimes \pi_{S^1\times S^1}^*H]-\pi_M^*[E_0])\right)\right)\\
&=\mu_{geo}[M_0\times S^1\times S^1,\pi_M^*E_0\otimes \pi_{S^1\times S^1}^*H,\phi_0\circ \pi_M,\nu\circ \pi_M,\eta\circ \pi_M]\\
&\qquad\qquad\qquad -\mu_{geo}[M_0\times S^1\times S^1,\pi_M^*E_0,\phi_0\circ \pi_M,\nu\circ \pi_M,\eta\circ \pi_M].
\end{align*}
Since the cycle $(M_0\times S^1\times S^1,\pi_M^*E_0,\phi_0\circ \pi_M,\nu\circ \pi_M,\eta\circ \pi_M)$ is null bordant, the proof is complete once providing a bordism 
$$(M_0\times S^1\times S^1,\pi_M^*E_0\boxtimes \pi_{S^1\times S^1}^*H,\phi_0\circ \pi_M,\nu\circ \pi_M,\eta\circ \pi_M)\sim_{bor} (M_0,E_0,\phi_0,\nu,\eta)^{1_\C},$$
where $1_\C\to M_0$ is the trivial complex line bundle. We let $\tilde{\pi}_M:M\times(S^1\times D)\setminus B_\epsilon\to M$ and $\tilde{\pi}_{S^1\times D}:M\times (S^1\times D)\setminus B_\epsilon\to (S^1\times D)\setminus B_\epsilon$ denote the projections. The bordism is provided by
$$(M_0\times (S^1\times D)\setminus B_\epsilon,\tilde{\pi}_M^*E_0\otimes \tilde{\pi}_{S^1\times D}^*\bar{H},\phi_0\circ \tilde{\pi}_M,\nu\circ \tilde{\pi}_M,\eta\circ \tilde{\pi}_M)$$

\end{proof}

\begin{remark}
In the extreme case $\He=0$, so $P=X$, the interested reader can construct an identification of both the twisted geometric $K$-homology and projective $K$-homology with geometric $K$-homology as defined by Baum-Douglas \cite{baumdtva}. After implementing a bordism\footnote{This bordism is only needed for the odd cycles and is constructed from the bordism $(S^1\times S^1,[H])\sim_{bor} (pt,\C)^{1_\C}$ considered in the proof of Theorem \ref{geometricmodification}.}, the geometric modification mapping coincides with the identity after these identifications.
\end{remark}

\begin{cor}
\label{assemblycorollary}
Let $P\to X$ be a principal $PU(\He)$-bundle.
\begin{enumerate}
\item When $\He=\C^n$ and $X$ is a finite $CW$-complex, $\mu_{geo}$ is an isomorphism if and only if $\delta:K^{geo}_*(P)\to K_*^{proj}(P)$ is an isomorphism. 
\item The geometric modification mapping $\delta:K^{geo}_*(P)\to K_*^{proj}(P)$ is an injection if and only if $\mu_{geo}$ is an injection.
\item The projective assembly mapping $\mu:K_*^{proj}(P)\to K_*^{an}(P)$ is a surjection if and only if $\mu_{geo}$ is a surjection.
\end{enumerate}
Furthermore, if $X$ is a smooth manifold, $\mu$ is surjective and $\delta$ is injective.
\end{cor}

\begin{proof}
That $\delta$ is an isomorphism for principal $PU(n)$-bundles over finite $CW$-complexes if and only if $\mu_{geo}$ is an isomorphism follows from Theorem \ref{geometricmodification} and Theorem \ref{analyticassembly}, since $\delta^{-1}=\mu_{geo}^{-1}\circ\mu$ and $\mu_{geo}^{-1}=\delta^{-1}\mu^{-1}$. Similarly, as $\mu\circ\delta=\mu_{geo}$, $\delta$ is injective if and only if $\mu_{geo}$ is  and $\mu$ surjective if and only if $\mu_{geo}$ is. The final statement of the corollary concerning smooth manifolds follows from the previous observations and \cite[Theorem 6.4]{wangcycles}.
\end{proof}

\subsection{Geometric $T$-duality for geometric cycles}

In Subsection \ref{tdualprojective}, we considered $T$-dual projective cycles. This construction can also be done but with target twisted geometric $K$-homology. 

\begin{deef}[The twisted geometric $T$-dual of geometric cycle on a circle bundle]
If $(M,E,\phi)$ is a geometric cycle on the total space of a circle bundle $Z\to X$ we define its twisted geometric $T$-dual cycle $(M,E,\phi)\rtimes_{geo} \R$ on $P^T$ by
\[(M,E,\phi)\rtimes_{geo} \R:=(M\times S^1,E,\pi_Z\circ \phi\times \id,\nu_{M\times S^1},\eta_{\tilde{\phi}}),\]
where 
\begin{enumerate}
\item We identify $E$ with its pullback to $M\times S^1$.
\item The mapping $\nu_{M\times S^1}:M\times S^1\to BSO(\infty)$ comes from the spin$^c$-structure on $M$ and the standard spin$^c$-structure on $S^1$.
\item The homotopy $\eta_{\tilde{\phi}}$ is constructed from the spin$^c$-structure on $M$ and the $PU(\He)$-equivariant mapping $\tilde{\phi}:M\times S^1\times PU(\He)\to P^T$ of Proposition \ref{tildephi} following Proposition \ref{hompropstab}.
\end{enumerate}
\end{deef}

\begin{prop}
If $Z\to X$ is a circle bundle $T$-duality produces a well defined monomorphism $\rtimes_{geo} \R:K_*^{geo}(Z)\to K_*^{geo}(P^T)$ that commutes with assembly mappings. So the following diagram commutes;
\[\begin{CD}
K_*^{geo}(Z)@>\rtimes_{geo}\R>>K_*^{geo}(P^T) \\
@V\mu VV @VV\mu_{geo} V \\
K_*^{an}(Z)@>\alpha_{CT}^{-1}>> K_*^{an}(P^T) \\
\end{CD}\]
In particular, the assembly mapping in twisted geometric $K$-homology is a surjection for $P^T$.
\end{prop}

The fact that the assembly mapping is a surjection in this case also follows from Corollary \ref{tdualsur} and Corollary \ref{assemblycorollary}. The proof of this Proposition goes along the same lines as Proposition \ref{tdualprop}, so we shall not prove this. The difference being that the equivariance property is hidden away in the homotopy $\eta_{\tilde{\phi}}$.

%
%
%
%
%
%
%

\Large
\section{Fractional index theory}
\label{fractionalsection}
\normalsize

One of the reasons to introduce geometric $K$-homology was the need for a more geometric homology theory describing the Atiyah-Singer index theorem. In this section we will study how one can use projective $K$-homology to describe the fractional index in a geometric way.  As mentioned in the introduction, this problem was the starting point for considering the construction of projective  cycles.

\subsection{The fractional index}

In this subsection, we take a closer look at the fractional index. Let us first recall the motivation coming from the fractional analytic index of projective pseudo-differential operators. If $A$ is an elliptic projective pseudo-differential operator, see Subsection \ref{clutching} or \cite{mms}, with parametrix $R$, the projective operator $[A,R]$ is a projective smoothing operator and one defines the fractional analytic index of $A$ by
\[\ind_{a}(A):=\tra[A,R]\in \C.\]
The trace functional $\tra$ is defined in \eqref{tracedefinition} (on page \pageref{tracedefinition}). The number $\ind_a(A)$ does not depend on the choice of parametrix because $\tra$ has the tracial property. A priori, the fractional analytic index is only an invariant of $A$ which takes complex values. However, much more can be said about this invariant.

\begin{thm}[Theorem $4$ of \cite{mms}]
\label{fractionalindextheorem}
The fractional analytic index of an elliptic projective pseudo-differential operator $A$ over a principal $PU(n)$-bundle is fractional and given by the formula
\[\ind_a(A)=\int_{T^*X}\ch_{\pi^*P}[\sigma(A)]\wedge Td(X).\]
\end{thm}

Motivated by Theorem \ref{fractionalindextheorem}, we define the fractional index $\ind_{f}:K^*_{proj}(P)\to \Q$ for $P\to X$ a smooth principal $PU(n)$-bundle over a closed manifold by
\[\ind_{f}(E):=\int_X \ch_P[E]\wedge \hat{A}(X).\]
So the content of Theorem \ref{fractionalindextheorem} is that $\ind_a(A)=\ind_f[\sigma(A)]$ since $\hat{A}(T^*X)=Td(X)$.

\begin{deef}[The projective index]
\label{projectiveindex}
Let $P\to X$ be a principal $PU(n)$-bundle and $(M,E,\phi)$ be a projective cycle over $P$. We define the projective index of $(M,E,\phi)$ via
\[\ind_{proj}(M,E,\phi):=\int_{PU(k)\backslash M}\ch_{\phi_0^*P}[\Psi_\phi E]\wedge\hat{A}(PU(k)\backslash M)\in \R,\]
where $\Psi_\phi$ denotes the stable isomorphism $\phi_0^*P\cong M$ associated with $\phi$.
\end{deef}

Our aim is of course to show that the projective index gives a well defined mapping on projective $K$-homology. It is instructive to consider the situation when $P\to X$ is a smooth principal $PU(n)$-bundle with a compatible spin$^c$-structure. Every cycle $(M,E,\phi)$ can be uniquely represented by the cycle $(P,\phi_!E,\id)$ in $K_*^{proj}(P)$ because of Poincar\'e duality. So we must show that the projective index of the two coincide. 

From the twisted Riemann-Roch Theorem, see Theorem $5.3$ and the remarks thereafter in \cite{careywang}, we have that 
\begin{align*}
\ind_{proj}(M,E,\phi)&=\int_{PU(k)\backslash M}\ch_{\phi_0^*P}[\Psi_\phi E]\wedge\hat{A}(PU(k)\backslash M)=\\
&=\int_{X}(\phi_0)_*\left(\ch_{\phi_0^*P}[\Psi_\phi E]\wedge\hat{A}(PU(k)\backslash M)\right)=\\
&=\int_X\ch_P(\phi_0)_!\Psi_\phi[E]\wedge \hat{A}(X)=\int_X\ch_P\phi_![E]\wedge \hat{A}(X)=\\
&=\ind_{proj}(P,\phi_!E,\id).
\end{align*}

\begin{prop}
\label{projindexprop}
For any principal $PU(n)$-bundle, the projective index gives a well defined mapping
\[\ind_{proj}:K_*^{proj}(P)\to \Q.\]
\end{prop}

\begin{proof}
Since Chern characters on the projective $K$-theory of a principal $PU(n)$-bundle takes values in rational cohomology, it is clear that the projective index takes rational values. We must show that the projective index respects the three relations in $K_*^{proj}$. It is obvious that the projective index respects direct sums. 

To prove that the projective index respects vector bundle modification, assume that $V_0\to PU(k)\backslash M$ is an even-dimensional spin$^c$ vector bundle and as usual let $V:=\pi^*V_0$. Then, by the twisted Riemann-Roch Theorem (Theorem $5.3$ in \cite{careywang}),
\begin{align*}
\ind_{proj}(M,E,\phi)^V&=\int_{S(V_0\oplus 1)} \ch_{(\pi^V)^*\phi^*P}\left[\Psi_{\phi\circ \pi^V}s_! E\right]\wedge \hat{A}(S(V_0\oplus 1))=\\
&= \int_{S(V_0\oplus 1)} \ch_{(\pi^V)^*\phi^*P}\left[s_! \Psi_{\phi}E\right] \wedge\hat{A}(S(V_0\oplus 1))=\\
&= \int_{S(V_0\oplus 1)} s_*\left(\ch_{\phi^*P}\left[\Psi_{\phi}E\right]\wedge \hat{A}(PU(k)\backslash M)\right)=\ind_{proj}(M,E,\phi).
\end{align*}

To verify that the projective index respects bordism, we observe that if $(W,E,\phi)$ is a projective cycle with boundary, $\Psi_\phi[E]|_{\partial W}=\Psi_{\phi|_{\partial W}}[E|_{\partial W}]$ and 
\[\hat{A}(W/PU(k))|_{\partial W/PU(k)}=\hat{A}(\partial W/PU(k)).\] 
Stokes Theorem implies that 
\begin{align}
\label{bordismintegral}
\nonumber
\ind_{proj}(\partial W,E|_{\partial W},\phi|_{\partial W})&= \int_{\partial W/PU(k)} \ch_{\phi_0^*P}[\Psi_\phi E]\wedge\hat{A}(W/PU(k))=\\
&=\int_{W/PU(k)} \rd\left(\ch_{\phi_0^*P}[\Psi_\phi E]\wedge\hat{A}(W/PU(k))\right)=0,
\end{align}
Note that, since $P$ is a $PU(n)$-bundle, the Chern character of a projective bundle is a closed form. 
\end{proof}

In the proof that the projective index respects bordism, we see the assumption that $P$ must be a principal $PU(n)$-bundle. This observation will be considered in more detail in the next section.

\begin{deef}[The twisted geometric index]
Let $P\to X$ be a principal $PU(n)$-bundle and $(M,E,\phi,\nu,\eta)$ be a twisted geometric cycle. We define the twisted geometric index of $(M,E,\phi,\nu,\eta)$ as
\[\ind_{geo}(M,E,\phi,\nu,\eta):=\int_{M}\ch_{\phi_0^*P\cdot Fr(-\nu)}[\Psi_\eta E]\wedge\hat{A}(M)\in \R,\]
where $\Psi_\eta$ denotes the stable isomorphism $\phi_0^*P\cdot Fr(-\nu)\cong M$ associated with $\eta$.
\end{deef}

\begin{prop}
If $P$ is a $PU(n)$-bundle, the twisted geometric index gives a well defined mapping
\[\ind_{geo}:K_*^{geo}(P)\to \Q.\]
\end{prop}

\begin{proof}
The proof goes along the same lines as Proposition \ref{projindexprop}. We need to show that the three relations are respected. It is obvious that the projective index respects direct sums. 

To prove that the twisted geometric index respects vector bundle modification, assume that $V\to M$ is an even-dimensional spin$^c$ vector bundle and $(M,E,\phi,\nu,\eta)$ a twisted geometric cycle. Then, by the twisted Riemann-Roch Theorem (Theorem $5.3$ in \cite{careywang}),
\begin{align*}
\ind_{geo}(M,E,\phi,\nu,\eta)^V&=\int_{S(V\oplus 1)} \ch_{(\pi^V)^*\phi_0^*P\cdot Fr(-\nu^V)}\left[\Psi_{\eta^V}s_! E\right]\wedge \hat{A}(S(V\oplus 1))=\\
&= \int_{S(V\oplus 1)} \ch_{(\pi^V)^*\phi^*_0P\cdot Fr(-\nu^V)}\left[s_! \Psi_{\eta}E\right] \wedge\hat{A}(S(V\oplus 1))=\\
&= \int_{S(V\oplus 1)} s_*\left(\ch_{\phi^*_0P\cdot Fr(-\nu)}\left[\Psi_{\eta}E\right]\wedge \hat{A}(M)\right)=\ind_{geo}(M,E,\phi).
\end{align*}

The bordism invariance of the twisted geometric index is proved exactly as in Proposition \ref{projindexprop}.  It again uses the fact that, when $P$ is a $PU(n)$-bundle, the Chern character maps to closed form.
\end{proof}

\begin{thm}
\label{geoprojfrac}
The geometric, projective and fractional index coincide in the following sense. If $P\to X$ is a principal $PU(n)$-bundle the following diagram commutes:
\[
\xymatrix@C=2.4em@R=3.71em{
&& K_*^{proj}(P)\ar[dd]^{\ind_{proj}}&&\\
K^*_{proj}(T^*_\perp P)\ar[urr]^{c}\ar[drr]_{\ind_{f}}  && &&K_*^{geo}(P)\ar[ull]_{\delta}\ar[dll]^{\ind_{geo}}\\
&&\Q& &}
\]
\end{thm}

\begin{proof}
If $(M,E,\phi,\nu,\eta)$ is a geometric cycle the twisted Riemann-Roch Theorem \cite[Theorem 5.3]{careywang} implies
\begin{align*}
\ind_{proj}\left(\delta(M,E,\phi,\nu,\eta)\right)&=\int_{S(T^*M\oplus 1)} \ch_{(\hat{\eta}\circ \pi_\nu)^*P}[\Psi_{\phi_{\nu\eta}}s_!E]\wedge \hat{A}(S(T^*M\oplus 1))=\\
&=\int_{S(T^*M\oplus 1)}\ch_{(\hat{\eta}\circ \pi_\nu)^*P}[s_!\Psi_{\eta}E]\wedge \hat{A}(S(T^*M\oplus 1))=\\
&=\int_{S(T^*M\oplus 1)}s_*\left(\ch_{\phi^*P\cdot Fr(-\nu)}[\Psi_{\eta}E]\wedge \hat{A}(M)\right)=\\
&=\ind_{geo}(M,E,\phi,\nu,\eta).
\end{align*}

That the left hand side commutes is a consequence of the twisted Riemann-Roch Theorem and the fact that 
\[[cl(A)]=[S(T^*_\perp P\oplus 1),s_![\sigma(A)],\phi],\]
as in \cite{baumdtva}. The mapping $s:T^*_\perp P\to S(T^*_\perp P\oplus 1)$ is defined as the obvious embedding.

\end{proof}

\subsection{The problem with non-torsion Dixmier-Douady invariants}
\label{nontorsionproblems}

An interesting question is whether fractional index theory exists for general $PU(\He)$-bundles. There is not really an obvious index formula in this setting. In particular, the twisted Chern character does not take values in a cohomology group that is isomorphic to de Rham cohomology, so we can not integrate. This can be explained by the fact that we can not push forward to a point since this map does not respect the twist. 

On the analytic side there are problems in constructing a fractional index for infinite-dimensional $\He$ as there exists obstructions to defining projective pseudo-differential operators. It follows from \cite[Theorem 8.1.1]{deformationgerbes} that it is not possible to construct projective pseudo-differential operators over a principal $PU(\He)$-bundle when the Dixmier-Douady invariant is non-torsion. To be more precise, one can not construct a symbolic calculus in this case. It is however possible to construct a pseudo-differential calculus for describing the twisted index pairing, see \cite{gofftwisted}.

Similarly, there are problems in having a well defined fractional index in the geometric models for twisted $K$-homology. We will now discuss these problems. As the same type of problems arise in both geometric twisted $K$-homology as in projective $K$-homology, we focus on the latter. 

\begin{lem}
\label{blaaa}
Let $M\to X$ be a principal $PU(n)$-bundle, $P\to X$ be a principal $PU(\He)$-bundle and $\phi:M\to P$ be a stable isomorphism. Then, there is an even degree form $\omega_\phi$ (uniquely determined modulo exact forms by $\phi$) making the following diagram commutative:
\[\begin{CD}
K^*_{proj}(M)@>\Psi_\phi>>K^*_{proj}(P) \\
@V\ch_MVV @VV\ch_P V \\
H^*_{dR}(X)@>\e^{\omega_\phi}>> H^*_{\Omega(P)}(X) \\
\end{CD}.\]
\end{lem}

\begin{proof}
By \cite{mathstev} the twisted Chern character is a complex isomorphism so $\e^{\omega_\phi}$ exists. If we have two such isomorphisms $\e^{\omega_\phi}$ and $\e^{\omega_\phi'}$, the form $\e^{\omega_\phi-\omega_\phi'}$ is closed and its cohomology class is $1$, that is $\omega_\phi-\omega_\phi'$ is exact.
\end{proof}

From Lemma \ref{blaaa}, we see exactly why the definition of the projective index breaks down for non-torsion Dixmier-Douady invariant. Namely, the mapping $\Psi_\phi$ is in general implemented by a non-closed form. Thus the bordism relation need not be respected by the projective index for non-torsion twists.

Being that it is the mapping $\Psi_\phi$ that causes problems with the projective index let us consider a definition of an index where $\Psi_\phi$ is not included. We define the naive projective index
\[\widetilde{\ind}_{proj}(M,E,\phi):=\int_{PU(k)\backslash M}\ch_{M}[E]\wedge\hat{A}(PU(k)\backslash M)\in \R.\]
We make the following observation based on the proof of Proposition \ref{projindexprop}.

\begin{prop}
The naive projective index respects direct sum/disjoint union and bordism relation but in general not vector bundle modification. 
\end{prop}

This can be seen already when there is no twist present since the naive projective index does not depend on the mapping part of the cycle.  A concrete example can be constructed by taking the space $\C P^2$ with no twist at all. The cycle $(\C P^2,\C P^2\times \C,\id)$ is a projective cycle over $\C P^2$ whose naive projective index is 
\[\widetilde{\ind}_{proj}(\C P^2,\C P^2\times \C,\id)=\hat{A}(\C P^2)=-\frac{1}{8}.\]
However, it is clear that $\widetilde{\ind}_{proj}(\C P^2,\C P^2\times \C,\id)$ does not depend on the mapping $\id:\C P^2\to \C P^2$ and so it can be computed from the cycle $(\C P^2,\C P^2\times \C)$ over a point. If the naive projective index were to respect vector bundle modification, it would only depend on the class of $(\C P^2,\C P^2\times \C)$ in $K_*^{geo}(pt)$. But $K_*^{geo}(pt)\cong \Z$ and the isomorphism is given by taking the index of the spin$^c$ Dirac operator. Hence $(\C P^2,\C P^2\times \C)$ is equivalent to $(pt,pt\times \C)$ in $K_*^{geo}(pt)$. However;
\[\widetilde{\ind}_{proj}(pt,pt\times \C)=1.\]
We can conclude that the naive projective index is not preserved under vector bundle modification.

\subsection{The associated rational classes}

We will in this subsection make the standard assumption that $P\to X$ is a principal $PU(n)$-bundle. It follows from the rational Chern character being an isomorphism that $K_*^{proj}(P;\Q)\cong K_*(X;\Q)$. We construct this isomorphism explicitly. To begin, the reader should recall that the group $K_*(X;\Q)$ can be realized using cycles of the following form: 

\begin{deef}
Let $X$ be a compact Hausdorff space and denote by $D$ the UHF-algebra with $K_0$-group the rational numbers (compare to \cite[Exercise 23.15.6]{blackadar}).  A rational $K$-cycle over $X$ is a triple, $(M,\xi,f)$, where $M$ is a smooth compact spin$^c$-manifold, $\xi \in K^0(X;\Q)\cong K^0(C(X)\otimes D)$, and $f:M\rightarrow X$ is a continuous map.
\end{deef}

Such cycles are defined using $K$-theory classes rather than bundles.  As such, there are only two relations, bordism and vector bundle modification. We will let the set of equivalence classes of rational $K$-cycles over $X$ under the relation generated by bordism and vector bundle modification be denoted by $K_*^{geo}(X;\Q)$; it forms an abelian group under disjoint union. It does in particular hold that $K_*^{geo}(X;\Q)$ is a module for the $\Q$-algebra $K^*(X;\Q)$, which is a unital algebra if $X$ is compact. The interested reader can find more details on these cycles and associated geometric realization of $KK^*(C(X),D)\cong K_*^{geo}(X;\Q)$ in \cite{wal}. We will use $E^n$ to denote the $n^\mathrm{th}$ iterated tensor product. For a $K$-theory class $x\in K^0(X)$ we let $x_\Q$ denote the image of $x$ under the ring homomorphism $K^0(X)\to K^0(X;\Q)$, induced by the unital inclusion $\C\to D$.

\begin{prop}
If $E$ is a vector bundle on a $d$-dimensional closed manifold $X$ whose rank has an $n^\mathrm{th}$ root, then the rational $K$-theory class $[E]_\Q\in K^0(X;\Q)$ admits an $n^\mathrm{th}$ root;
\[\sqrt[n]{[E]_\Q}:=\sqrt[n]{rk(E)}+\sum_{k=1}^d \begin{pmatrix} 1/n\\k\end{pmatrix} \left([E]_\Q-rk(E)[1]\right)^k\in K^0(X;\Q).\]
The $n^\mathrm{th}$ root satisfies $[E]_\Q=\sqrt[n]{[E]_\Q}^n$. Any vector bundle $F$ satisfies $[F]_\Q=\sqrt[n]{[F]^n_\Q}$ in $K^0(X;\Q)$. Furthermore, roots commutes with Chern characters.
\end{prop}

We leave the details to the reader who should note the Taylor expansion:
\[(1+x)^\alpha=\sum_{k=0}^\infty \begin{pmatrix} \alpha\\k\end{pmatrix} x^k, \quad\mbox{for}\quad|x|<1.\]

If $E$ is a projective vector bundle over a principal $PU(n)$-bundle, we define its rationalization $[E]_\Q$ as the rational $K$-theory class
\[[E]_\Q:=\sqrt[n]{E^n}.\]
This definition produces an element of $K^*(X;\Q)$ using the fact that $E^n$ descends to a vector bundle on $X$. The rational class $[E]_\Q\in K^0(X;\Q)$ of a projective bundle is well defined by Proposition \ref{elementpropproj}.

\begin{prop}
\label{rationalktheory}
Let $P\to X$ denote a principal $PU(n)$-bundle on a closed manifold. Rationalization of projective bundles induces a well defined mapping $K^*_{proj}(P)\to K^*(X;\Q)$ which is a rational isomorphism.
\end{prop}

\begin{proof}
Since $\ch[E]_\Q=\ch_P[E]$ and the Chern character is a rational isomorphism, $[E]\mapsto [E]_\Q$ is both additive and a rational isomorphism.
\end{proof}

If $M$ is a spin$^c$-manifold we let $L^{W_3}(M)\to M$ denote the determinant line bundle associated with the spin$^c$-structure. To be precise, if $Fr^c(M)\to M$ denotes the principle spin$^c$-bundle that lifts the frame bundle $Fr(M)\to M$, the determinant bundle is $Fr^c(M)\times_{Spin^c}U(1)$.

\begin{prop}
In the notation of the previous paragraph, we have 
\[\ch[\sqrt{L^{W_3}(M)}]_\Q \wedge Td(M)=\hat{A}(M).\]
\end{prop}

\begin{proof}
The proof of the Proposition follows by observing $\e^{c_1(L^{W_3}(M))/2}\wedge Td(M)=\hat{A}(M)$.
\end{proof}

Recall the notation $\Psi_\phi$ for the stable isomorphism $\phi_0^*P\cong M$ associated with a stable morphism $\phi$ from Definition \ref{projectiveindex}.

\begin{deef}
\label{geofromproj}
If $(M,E,\phi)$ is a projective cycle we define its rationalization as the geometric rational $K$-cycle over $X$ given by
\[(M,E,\phi)_\Q=\left(PU(k)\backslash M,[\Psi_\phi E]_\Q\otimes \sqrt{L^{W_3}(PU(k)\backslash M)},\phi_0\right).\]
\end{deef}

\begin{prop}
Rationalization of projective cycles induces a well defined mapping $R:K_*^{proj}(P)\to K_*(X;\Q)$ which is a rational isomorphism.
\end{prop}

\begin{proof}
It is clear from Proposition \ref{rationalktheory} and the fact that the assembly mapping is an isomorphism, that if the mapping is well defined, it is a rational isomorphism. That the rationalization is well defined is a standard verification that we leave to the reader.
\end{proof}

\begin{thm}
\label{rationalization}
If $P\to X$ is a principle $PU(n)$-bundle, the following diagram commutes:
\[\xymatrix@C=2.4em@R=3.71em{
K_*^{proj}(P)\ar[rr]^{R}\ar[dr]_{\ind_{proj}}  && K_*(X;\Q)\ar[dl]^{\ind_\Q}\\
& \Q&}\]
\end{thm}

\begin{proof}
We have that 
\begin{align*}
\ind_\Q&\left([L]\cap (M,E,\phi)_\Q\right)=\\
&=\int_{PU(k)\backslash M} \ch\left([L^{W_3}(PU(k)\backslash M)]_\Q\cup [\Psi_\phi E]_\Q\right) \wedge Td(PU(k)\backslash M)=\\
&=\int_{PU(k)\backslash M}\ch_{\phi_0^*P}\left(\Psi_\phi [E]\right)\wedge \hat{A}(PU(k)\backslash M)=\ind_{proj}(M,E,\phi).
\end{align*}
\end{proof}

\paragraph{\textbf{Acknowledgements}}
The authors wish to express their gratitude towards P. Baum, A. Carey and B.-L. Wang for discussions and access to \cite{BCW}. The authors also wish to thank V. Mathai for inspiring discussions in the early stages of this project. The first listed author was supported by an NSERC postdoctoral fellowship. The second author gratefully acknowledges support by the Institut Mittag-Leffler (Djursholm, Sweden). Both authors thank the Courant Centre of G\"ottingen, the Leibniz Universit\"at Hannover and the Graduiertenkolleg 1463 (\emph{Analysis, Geometry and String Theory}) for facilitating this collaboration.

\end{document}